\documentclass[11pt,oneside]{amsart}
\usepackage[pagebackref,breaklinks,unicode]{hyperref}
\hypersetup{pdftitle = Coarse obstructions to cocompact cubulation}
\title{Coarse obstructions to cocompact cubulation}

\author{Zachary Munro}
\address{Israel Institute of Technology, Department of Mathematics, Haifa 3200003, Israel}
\email{munrozachary@campus.technion.ac.il}

\author{Harry Petyt}
\address{Mathematical Institute, University of Oxford, UK}
\email{petyt@maths.ox.ac.uk}

\usepackage{amsmath, amssymb, amsthm} 
\usepackage[english]{babel} 
\usepackage[font=small,justification=centering]{caption,subcaption} 
\usepackage[nodayofweek]{datetime}
\usepackage{enumitem} \setlist{nosep} 
\usepackage[T1]{fontenc} 
\usepackage[a4paper]{geometry} 
\usepackage[utf8]{inputenc} 
    \usepackage{csquotes} 
\usepackage{ifthen} 
\usepackage{mathabx} 
\usepackage{mathtools} 
\usepackage[dvipsnames]{xcolor} 
\usepackage[pagebackref,breaklinks,unicode]{hyperref} 
    \renewcommand*{\backrefalt}[4]{\ifcase #1 (Not cited).\or (Cited p.~#2).\else (Cited pp.~#2).\fi} 
\usepackage{setspace} 
\usepackage{textcomp} 
\usepackage{tikz-cd} 
\usepackage{xfrac} 

\usepackage{import}

\let\OLDthebibliography\thebibliography \renewcommand\thebibliography[1]{   
    \OLDthebibliography{#1}\setlength{\parskip}{0pt}\setlength{\itemsep}{0pt plus 0.3ex}} 

\makeatletter\def\subsection{\@startsection{subsection}{1}\z@{.7\linespacing\@plus\linespacing}
    {.5\linespacing}{\normalfont\scshape\centering}}\makeatother 

\newcounter{shcount}
\newcounter{thmcount}
\newcounter{enumlabelcount}
\newcounter{claimcount}

\newcommand*{\bsh}[1]{\theoremstyle{definition}\newtheorem{subhead\theshcount}[theorem]{#1}
    \begin{subhead\theshcount}} 
\newcommand*{\esh}{\end{subhead\theshcount}\stepcounter{shcount}} 
\newcommand*{\ubsh}[1]{\theoremstyle{definition}\newtheorem*{subhead\theshcount}{#1}
    \begin{subhead\theshcount}} 
\newcommand*{\uesh}{\end{subhead\theshcount}\stepcounter{shcount}} 

\newcommand*{\numberedtheorem}[3]{\theoremstyle{plain}\newtheorem*{makethm\thethmcount}{#1}
    \ifthenelse{\equal{#2}{}}{\begin{makethm\thethmcount}#3\end{makethm\thethmcount}\stepcounter{thmcount}}
    {\begin{makethm\thethmcount}[#2]#3\end{makethm\thethmcount}\stepcounter{thmcount}}} 

\makeatletter\newcommand\enumlabel[1][]{\item[#1]
    \refstepcounter{enumlabelcount}\def\@currentlabel{#1}}\makeatother

\newenvironment{claim*}{\medskip\noindent\textbf{Claim:}\hspace{0.5mm}}{}

\newenvironment{claim*proof}{\medskip\noindent\emph{Proof of Claim.}\hspace{0.5mm}}
    {\leavevmode\unskip\penalty9999\hbox{}\nobreak\hfill\quad\hbox{$\diamondsuit$}\medskip}

\renewcommand{\th}{\ensuremath{^\mathrm{th}}}
\renewcommand{\bf}{\mathbf}

\newcommand*{\eps}{\varepsilon}

\newcommand*{\R}{\mathbf{R}}

\newcommand*{\Z}{\mathbf{Z}}

\newcommand*{\cal}{\mathcal}

\renewcommand*{\i}{\mathbf{i}}

\newcommand*{\ssm}{\smallsetminus}
\renewcommand*{\t}{\mathbf{t}}

\DeclareMathOperator{\Aut}{Aut}

\DeclareMathOperator{\dist}{\mathsf{d}}
\DeclareMathOperator{\gd}{gd}

\DeclareMathOperator{\qfrk}{qf.rk}

\DeclareMathOperator{\rk}{rk}

\DeclareMathOperator{\sepdim}{sepdim}

\DeclareMathOperator{\stab}{Stab}
\DeclareMathOperator{\supp}{Supp}

\DeclareMathOperator{\vcd}{vcd}
\DeclareMathOperator{\vgd}{vgd}

\newcommand{\ignore}[2]{\left\{\kern-.7ex\left\{#1\right\}\kern-.7ex\right\}_{#2}}

\newcommand*{\sgen}[1]{\langle#1\rangle}

\newcommand*{\mk}{\medskip}

\definecolor{harrycomment}{rgb}{0.6,0,0.4}

\definecolor{zachcomment}{rgb}{0.55,0.71,0}

\newtheorem{mthm}{Theorem} 
\newtheorem{question}{Question}
\swapnumbers
\newtheorem{theorem}{Theorem}[section]
\newtheorem{lemma}[theorem]{Lemma} 
\newtheorem{corollary}[theorem]{Corollary}
\newtheorem{proposition}[theorem]{Proposition}
\theoremstyle{definition}
\newtheorem{definition}[theorem]{Definition}
\newtheorem{remark}[theorem]{Remark}

\newtheorem{example}[theorem]{Example}

\definecolor{markcomment}{rgb}{0,0.6,0.4}

\begin{document}

\begin{abstract}
We provide geometric methods to give bounds on the large-scale dimension of CAT(0) cube complexes quasiisometric to a given group $G$. In situations where these bounds conflict we obtain obstructions to $G$ being cocompactly cubulated. More strongly, the obstructions prevent $G$ from being a coarse median space.

As applications, we show that many free-by-cyclic groups cannot be cocompactly cubulated, even virtually, and prove that any tubular group with a coarse median is virtually compact special. We also exhibit a group that is CAT(0), $C(6)$, and virtually special, yet is not quasiisometric to any CAT(0) cube complex. This is the first example of a $C(6)$ group that cannot be cocompactly cubulated, resolving a question of Jankiewicz and partially answering a question of Wise. 
\end{abstract}

\maketitle

\section{Introduction} \label{sec:intro}

Cubulation has proved to be an important tool in the study of finitely generated groups. That is, when studying a group $G$, if one can find an action of $G$ on a CAT(0) cube complex, then one gains access to powerful combinatorial machinery controlling the geometry of $G$. 

The unqualified term \emph{cubulation} generally refers to a proper action of $G$ on a (possibly infinite-dimensional) CAT(0) cube complex; i.e. an action where each ball contains only finitely many orbit points (with multiplicity). Stronger conclusions can be drawn from a \emph{cocompact cubulation}, namely a proper cocompact action of $G$ on a (necessarily finite-dimensional) CAT(0) cube complex. And strongest of all is for $G$ to be \emph{virtually compact special}, meaning that a finite-index subgroup of $G$ is the fundamental group of a finite cube complex that is \emph{special} in the sense of \cite{haglundwise:special}.

Many groups have been successfully (cocompactly) cubulated, and many of these are even virtually special; indeed, Agol's theorem \cite{agol:virtual} states that every hyperbolic group that can be cocompactly cubulated is virtually compact special. However, there are many groups of interest for which cocompact cubulability is unknown. For many, it is expected to be impossible; we simply lack refined criteria to verify this suspicion. In fact, all general methods for obstructing cocompact cubulation known to the authors are simply negations of properties of cocompactly cubulated groups. For instance, if $G$ has super-quadratic Dehn function, property~(T), distorted elements, or elements whose centralisers do not virtually split, then $G$ is not cocompactly cubulated \cite{nibloreeves:groups,eberlein:canonical}. 

Our goal in this article is to provide the first general obstructions to cocompact cubulation that are not a negation of some group-theoretic property. These obstructions are of a coarse-geometric nature, and in fact they provide a strong negation to the possibility of cocompact cubulation, by ruling out the existence of a \emph{coarse median}.

If $G$ acts properly cocompactly on a CAT(0) cube complex $X$, then $G$ is quasiisometric to $X$. Merely being quasiisometric to a CAT(0) cube complex (\emph{quasicubical}) is much weaker than being cocompactly cubulated; for instance, all hyperbolic groups are quasicubical \cite{haglundwise:combination}, even though some have property~(T) and hence cannot act on CAT(0) cube complexes without global fixed-points \cite{nibloreeves:groups}. If $G$ is quasicubical then one can pass the median of the cube complex along the quasiisometry to equip $G$ with a ternary operator that behaves coarsely the same. This makes $G$ a \emph{coarse median space} in the sense of \cite{bowditch:coarse}. Thus proving that a space does not admit any coarse median is much stronger than proving it cannot be cocompactly cubulated.

Roughly speaking, for each $n\ge 2$ we describe geometric configurations whose ``geometric rank'' is $n$ but that cannot appear in any space of ``median rank'' less than $n+1$; see Figure~\ref{fig:RBF}. Showing that $G$ cannot have a coarse median then amounts to finding such a configuration in $G$ ``at top rank''. This can be summarised as follows (see Theorem~\ref{thm:rbf}). By the \emph{quasiflat rank} of a metric space $X$, we mean the supremal integer $\qfrk X$ for which there is a quasiisometric embedding $\R^{\qfrk X}\to X$.

\begin{mthm} \label{mthm:obstruction}
Let $G$ be a finitely generated group with $\qfrk G\le n$. If $G$ contains a quasiisometrically embedded \emph{richly branching flat} of dimension $n$, then $G$ cannot be a coarse median space. In particular, $G$ cannot be virtually cocompactly cubulated.
\end{mthm}

The motivating observation for defining richly branching flats (Definition~\ref{def:rbf}) is that in a, say, $2$--dimensional CAT(0) cube complex, half-flats can only branch off a flat in two directions, i.e. those parallel to the coordinate axes. This picture becomes muddled when considering spaces only up to quasiisometry. For example, a cyclic gluing of six quarter-planes is quasiisometric to a flat, and half-flats can branch off such an object in more than two directions. However, in a sense, this is the as complicated as it gets: quasiflats in a coarse median space are well-approximated by unions of orthants \cite{bowditch:quasiflats}. This quasiflat rigidity generalises (and strengthens) several previous theorems \cite{bestvinakleinersageev:quasiflats,huang:top,behrstockhagensisto:quasiflats}. 

As the name suggests, then, richly branching flats should be thought of as being flats with ``too much'' branching ``all over''. Similar configurations were considered in work of Haettel \cite{haettel:higher}, who characterised which symmetric spaces and affine buildings admit coarse medians (see Theorem~\ref{mthm:thomas}). However, his arguments rely on work specific to that setting, such as \cite{kleinerleeb:rigidity}.

\begin{figure}[ht] 
\includegraphics[width=6.5cm, trim = 0 3mm 0 3mm]{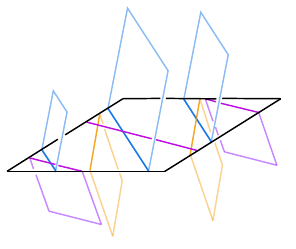} 
\caption{A 2--dimensional richly branching flat. Note that the only intersection between half-planes occurs in the base flat.} \label{fig:RBF} 
\end{figure}

Part of the proof of Theorem~\ref{mthm:obstruction} involves relating the quasiflat rank of a coarse median space $X$ to its \emph{coarse-median rank}, $\rk X$. Proposition~\ref{prop:rank} in particular shows that these two quantities agree when $X$ is a finitely generated group. In general it is not easy to ascertain the quasiflat rank of a group, but we nevertheless have the following as a consequence of Proposition~\ref{prop:rank}, which will be used in our applications of Theorem~\ref{mthm:obstruction} below. 

\begin{mthm} \label{mthm:vgd}
Let $G$ be a finitely generated group. If $G$ admits a coarse median, then $\rk G\le\vcd G$.
\end{mthm}

Although this bound is not optimal in general, as can be seen from hyperbolic manifolds, it can be useful in concrete situations where cohomological dimensions are easily computable.

\mk

Let us now discuss the applications of Theorem~\ref{mthm:obstruction} considered in this paper. The two main classes of groups that we consider are \emph{free-by-cyclic} groups and \emph{tubular} groups.

Free-by-cyclic groups form a heavily-studied class, with rich behaviours in accordance with the theory of free-group automorphisms. (In this paper, all free group kernels are assumed to be finitely generated.) Among the numerous known properties of free-by-cyclic groups, we find that they pass the ``obvious'' tests for cocompact cubulability: they have quadratic isoperimetric functions \cite{bridsongroves:quadratic} and their abelian subgroups are undistorted \cite{button:aspects}, for instance.

It is therefore natural to ask which free-by-cyclic groups can be cocompactly cubulated. As it turns out, all hyperbolic free-by-cyclic groups are cocompactly cubulated, by work of Hagen--Wise \cite{hagenwise:cubulating:irreducible,hagenwise:cubulating:general}. (More generally, all hyperbolic \emph{hyperbolic}-by-cyclic groups are \cite{dahmanikrishnamutanguha:hyperbolic}.) Also, Hagen--Przytycki characterised which graph manifold groups are cocompactly cubulated \cite{hagenprzytycki:cocompactly}, and some of those are free-by-cyclic. Gersten's group \cite{gersten:automorphism} (and tweaks thereof) does not act properly on a $\mathrm{CAT}(0)$ space, for reasons involving translation-length, and so cannot be cocompactly cubulated. There are also examples that can be cubulated, but not cocompactly \cite{wuye:some}. Though cocompact cubulability is expected to be rare amongst non-hyperbolic free-by-cyclic groups, little seems to be known in general. For example, the following question is open. 

\begin{question}
Are toral relatively hyperbolic free-by-cyclic groups cocompactly cubulated? 
\end{question}

Using Theorem~\ref{mthm:obstruction}, we show that many free-by-cyclic groups cannot be cocompactly cubulated; indeed they cannot even admit a coarse median. The following is Theorem~\ref{thm:two_linear}.

\begin{mthm} \label{mthm:fbz_no_median}
If a free-by-cyclic group has \emph{rich linearity},
then it  cannot admit a coarse median.
\end{mthm}

The rough idea of the rich linearity condition (see Definition~\ref{def:rich}) in Theorem~\ref{mthm:fbz_no_median} is as follows. The free-by-cyclic group $G=F\rtimes_\phi\Z$ virtually has an \emph{improved relative train track} structure (from \cite{bestvinafeighnhandel:tits:1}; see Item~\ref{sh:irtt}), which provides a systematic way to build up $G$ in layers known as \emph{strata}. The simplest strata are those that are fixed by $\phi$; these can give rise to $\Z^2$ subgroups of $G$. The translates of other strata by powers of $\phi$ grow in complexity. When this growth is linear, this can give rise to quasiflats that branch off the above $\Z^2$ subgroups. The rich linearity condition enforces that this happens enough to find a richly branching flat and employ Theorem~\ref{mthm:obstruction}.


The class of free-by-cyclic groups with rich linearity is larger than those for which Gersten's ideas rule out being CAT(0); a simple example that can be shown to be CAT(0) is provided in Example~\ref{eg:more_than_gersten}. 

In the positive direction, one could more leniently ask which free-by-cyclic groups are quasicubical. For instance, it turns out (either by \cite{kapovichleeb:3manifold} and \cite{hagenprzytycki:cocompactly}, or by \cite{hagenrussellsistospriano:equivariant} and \cite{petyt:mapping}) that all graph manifolds are quasicubical, even those that cannot be cocompactly cubulated. More generally, we show the following. See Proposition~\ref{prop:few_linear} for a more precise statement.

\begin{mthm} \label{mthm:at_most_one}
If a free-by-cyclic group $G$ is represented by an improved relative train track with no quadratically growing strata and such that each Nielsen cycle supports at most one linear stratum, then $G$ is virtually a colourable hierarchically hyperbolic group, and in particular is quasicubical.
\end{mthm}

For free-by-cyclic groups of linear growth, the combination of Theorems~\ref{mthm:fbz_no_median} and~\ref{mthm:at_most_one} leaves quite a restricted set of possible train tracks. The following is related to a possible converse to Theorem~\ref{mthm:fbz_no_median}; see Remark~\ref{sh:converse} for more discussion and related questions. Note that such a converse would imply that rich linearity is witnessed by all fibrations.

\begin{question} \label{qn:fbz_qq}
Is it true that a free-by-cyclic group is virtually hierarchically hyperbolic if and only if it does not have rich linearity? Less strictly, if a free-by-cyclic group has no quasiisometrically embedded richly branching flats, must it be quasiisometric to a finite-dimensional CAT(0) cube complex?
\end{question}

Our other main application of Theorem~\ref{mthm:obstruction} is to tubular groups. A group is \emph{tubular} if it can be written as a graph of groups with $\Z^2$ vertices and $\Z$ edges. This simple description belies the remarkably varied behaviour that tubular groups display. Indeed, the class includes: Gersten's group that is not a subgroup of any CAT(0) group \cite{gersten:automorphism}; CAT(0) groups that have quadratically diverging rays but no super-quadratically diverging rays \cite{gersten:quadratic}; non-Hopfian CAT(0) groups \cite{wise:nonhopfian}; Croke--Kleiner groups, which have CAT(0) structures with differing visual boundaries \cite{crokekleiner:spaces}; and Brady--Bridson groups, which fill the isoperimetric spectrum \cite{bradybridson:there}. 



We use richly branching flats to prove a rigidity result for cubulations of tubular groups (Theorem~\ref{thm:tubular_special}). 

\begin{mthm} \label{mthm:tubular_special}
A tubular group admits a coarse median if and only if it is cocompactly cubulated and virtually compact special.
\end{mthm}

In \cite{wise:cubular}, Wise characterised which tubular groups admit free actions on CAT(0) cube complexes. These cubulations were further investigated by Woodhouse \cite{woodhouse:classifying:finite,woodhouse:classifying:virtually}. Wise also showed that all cocompactly cubulated tubular groups are virtually compact special.

As an instance of this characterisation, we construct a tubular group that has a $C(6)$ structure but does not admit any coarse median (Example~\ref{eg:c6tubular}). This is the first example of a $C(6)$ group that does not act properly cocompactly on any $\mathrm{CAT}(0)$ cube complex, answering a question asked by Jankiewicz \cite[Q.~6.6.4]{jankiewicz:cubical} and suggested by Wise \cite{wise:cubulating}. 

\begin{mthm} \label{mthm:example}
There exists a group $G$ that is $C(6)$, $\mathrm{CAT}(0)$, and virtually special, yet does not admit any coarse median. In particular, $G$ is not virtually cocompactly cubulated. 
\end{mthm}

Theorem~\ref{mthm:example} contrasts with the situation for $C'(\frac16)$ groups, which are cocompactly cubulated \cite{wise:cubulating}. It is still open whether there is some $n$ for which all $C(n)$ groups are cocompactly cubulated. It also contrasts with the \emph{strict $C(6)$} condition introduced in \cite{munrowise:srict}: the authors show that strict $C(6)$ groups are hyperbolic relative to virtually-$\Z^2$ subgroups, and they are therefore quasicubical.

In their work on the isoperimetric spectrum \cite{bradybridson:there}, Brady--Bridson established a relationship between distortion and isoperimetry in a parametrised family of tubular groups with one vertex group and two edge groups. Along the way to Theorem~\ref{mthm:tubular_special}, we show the following (Theorem~\ref{thm:distortion}).

\begin{mthm} \label{mthm:tubular_distortion}
A tubular group contains a distorted element if and only if it has super-quadratic Dehn function.
\end{mthm}

\mk

Finally, as noted above, the configurations considered in \cite{haettel:higher} are instances of richly branching flats. One can therefore use the arguments of this paper to recover the following result.

\begin{mthm}[{\cite[Thm~C]{haettel:higher}}] \label{mthm:thomas}
Let $X$ be a symmetric space of non-compact type, or a thick affine building. There exists a coarse median on $X$ if and only if the spherical type of $X$ is $A_1^n$.
\end{mthm}

We observe that this result of Haettel answers negatively a question asked by Wise in the 2014 ICM proceedings \cite[Prob.~13.40]{wise:cubical}, because it provides CAT(0) spaces with proper cocompact group actions that are not quasiisometric to CAT(0) cube complexes. Theorem~\ref{mthm:example} provides a new example, and indeed one can use Theorem~\ref{mthm:tubular_special} to produce many such examples. 

\ubsh{Acknowledgements}
We are indebted to Naomi Andrew for several very useful discussions about free-by-cyclic groups. We are very grateful to Mark Hagen and Monika Kudlinska for helpful conversations about train tracks. We thank Jason Behrstock, Pritam Ghosh, Rob Kropholler, Jean Pierre Mutanguha, and Xiaolei Wu for comments on an earlier version. We also thank the organisers of the 2023 thematic program on geometric group theory in Montreal, where this work began.
\uesh

\section{Preliminaries} \label{sec:prelims}

\subsection{Ultralimits}

\begin{definition}[Ultrafilter]
An \emph{ultrafilter} $\omega$ on $\mathbf N$ is a set of subsets of $\mathbf N$ satisfying:
\begin{itemize}
    \item If $A\in \omega$ and $A\subset B$, then $B\in \omega$.
    \item If $A,B\in \omega$, then $A\cap B\in \omega$.
    \item For all $A\subset S$, either $A\in \omega$ or $S\ssm A\in \omega$. 
\end{itemize}
The ultrafilter $\omega$ is \emph{non-principal} if it contains no finite sets.
\end{definition}

One can think of sets in $\omega$ as having measure 1, and those not in $\omega$ as having measure 0. We shall work only with non-principal ultrafilters, so we refer to them simply as ``ultrafilters''.

\begin{definition}[Ultralimit]
Given a sequence $(x_n)\in \mathbf R$, if there exists $x\in \mathbf R$ such that $\{n : |x_n-x|<\epsilon\}\in \omega$ for every $\epsilon>0$, then we call $x$ the \emph{ultralimit} of $(x_n)$ and write $x=\lim_\omega x_n$. Every sequence has at most one ultralimit. Let $(X_n,\dist_n)$ be a sequence of metric spaces with basepoints $b_n\in X_n$. The \emph{ultralimit} $\lim_\omega(X_n,b_n)$ can be defined as the metric quotient of the pseudometric space whose elements are sequences $(x_n)$, with $x_n\in X_n$, such that $\lim_\omega\dist_n(x_n,b_n)$ exists, and whose pseudometric is $\hat\dist_\omega((x_n),(y_n))=\lim_\omega\dist_n(x_n,y_n)$.
\end{definition}

Every ultralimit of metric spaces is complete \cite[Lem.~I.5.53]{bridsonhaefliger:metric}. One often considers ultralimits where the terms in the sequence are all derived from the same metric space. For instance, if $X$ is a proper metric space, then $X$ can be written as the ultralimit of a sequence of nested balls; $X=\lim_\omega(B_X(b,n),b)$. Two other important cases are \emph{asymptotic cones} and \emph{tangent cones}.

\begin{definition}[Cones]
Let $X$ be a metric space, and let $b\in X$. Let $(\lambda_n)$ be a sequence of real numbers. If $\lambda_n\to 0$, then we call $\hat X=\lim_\omega(X,\lambda_n\dist,b)$ an \emph{asymptotic cone} of $X$; it is independent of the choice of $b$. If $\lambda_n\to\infty$, then we call $T_bX=\lim_\omega(X,\lambda_n\dist,b)$ a \emph{tangent cone} of $X$ at $b$.
\end{definition}

\begin{lemma} \label{lem:qi_induces_bilipschitz}
Let $X$ and $Y$ be metric spaces and let $(\lambda_n)$ be a sequence converging to 0. Any quasiisometric embedding $f:X\to Y$ induces a bilipschitz embedding of asymptotic cones $\hat f:\hat X\to\hat Y$.
\end{lemma}

\begin{proof}
Given $x=(x_n)\in\hat X$, let $\hat f(x)=(f(x_n))$. The map $\hat f$ is well defined, because if $(x_n)=(z_n)$, then $\lim_\omega\lambda_n\dist_X(x_n,z_n)=0$, and hence $\lim_\omega\lambda_n\dist_Y(f(x_n),f(z_n))=0$ because $f$ is a quasiisometry. Essentially the same argument shows that $\hat f$ is bilipschitz.
\end{proof}

\subsection{Medians}

\begin{definition}[Median algebra] \label{def:median}
A \emph{median algebra} is a set $M$ with a ternary operation $\mu$ satisfying:
\[
\mu(a,a,x)=a, \quad \mu(a,b,x)=\mu(a,x,b)=\mu(x,a,b), \quad \mu(a,b,\mu(x,y,z))=\mu(\mu(a,b,x),\mu(a,b,y),z)
\]
for all $a,b,x,y,z\in M$. The latter equality is called the \emph{five-point condition}.
\end{definition}

It can be useful to interchangeably think of  $\mu$ both as a median operator and as giving ``projection'' maps $\mu(a,b,\cdot)$ from $M$ to the ``hull'' of $\{a,b\}$. For instance, the five-point condition can (almost) be described by the slogan ``the projection of the median is the median of the projections''. 

For a subset $A$ of a median algebra $M$, let $J(A)=\{\mu(a,a',x)\,:\,a,a'\in A,\, x\in M\}$. The subset $A$ is \emph{median-convex} if $J(A)=A$. The median-convex hull of a subset $A$ is the intersection of all median-convex subsets containing $A$. If $M$ has rank $n$, then the median-convex hull of $A$ can be obtained as $J^n(A)$ (see \cite[Prop.~8.2.3]{bowditch:median:book}, for instance).

\begin{definition}[Median morphism, rank]
If $(M,\mu)$ and $(N,\nu)$ are median algebras, then a map $f:M\to N$ is a \emph{median morphism} if $f\mu(x,y,z)=\nu(fx,fy,fz)$ for all $x,y,z\in M$. The \emph{rank} of a median algebra $M$, denoted $\rk M$, is the supremal $n$ such that there is a median monomorphism $\{0,1\}^n\to M$.
\end{definition}

A \emph{wall} in a median algebra $M$ is a partition of $M$ into two nonempty, median-convex subsets, called \emph{halfspaces}. Two walls are said to \emph{cross} if all four \emph{quarterspaces} (intersections of halfspaces) are nonempty. According to \cite[Prop.~6.2]{bowditch:coarse}, the rank of a median algebra is equal to the supremal cardinality of a set of pairwise crossing walls.

\begin{definition}[Median metric space]
A metric space $(X,\dist)$ is a \emph{median metric space} if for every $x_1,x_2,x_3\in X$ there is a unique point $\mu$ such that $\dist(x_i,x_j)=\dist(x_i,\mu)+\dist(\mu,x_j)$ for all $i\ne j$.
\end{definition}

One basic example of a median metric space is a \emph{panel}, i.e. a direct product of a finite number of nontrivial closed intervals in $[0,\infty)$, equipped with the $\ell^1$-metric and thus the component-wise median.

It can be shown that the map $(x_1,x_2,x_3)\mapsto\mu$ makes $X$ into a median algebra \cite{sholander:medians}. Moreover, this map is 1--Lipschitz in each factor: we have $\dist(\mu(x,y,z),\mu(x,y,z'))\le\dist(z,z')$, and similarly for the other factors by symmetry of $\mu$. One simple consequence is the following, rather crude, estimate.

\begin{lemma} \label{lem:ball_hull}
Let $X$ be a median metric space of rank $n$. For every $x\in X$ and every $r$, the median-convex hull of the ball $B_X(x,r)$ is contained in the ball $B_X(x,2^nr)$.
\end{lemma}

\begin{proof}
It suffices to show that $J(B_X(x,r))\subset B_X(x,2r)$. But this holds because if $x_1,x_2\in B_X(x,r)$ and $z\in X$, then $\dist(x,\mu(x_1,x_2,z))\le\dist(x,\mu(x,x,z))+\dist(x,x_1)+\dist(x,x_2)\le2r$.
\end{proof}

Every complete, connected median metric space is geodesic (see \cite[Lem.~13.3.2]{bowditch:median:book}, for instance). Since the completion of any median metric space is also a median metric space, we shall always implicitly assume that our median metric spaces are complete.

We say that a subset $Y$ of a metric space $X$ is \emph{$r$--separated} if $\dist(y_1,y_2)\ge r$ for every $y_1,y_2\in Y$.

\begin{lemma} \label{lem:tangent_median}
If $(X,\mu)$ is a connected median metric space of rank $n$ and $p\in X$, then every tangent cone of $X$ at $p$ is a connected median metric space of rank at most $n$.
\end{lemma}

\begin{proof}
Let $(\lambda_n)$ be a sequence with $\lambda_n\to\infty$, and let $T_pX$ be the corresponding tangent cone of $X$ at $p$. Given points $x^1=(x^1_n)$, $x^2=(x^2_n)$, and $x^3=(x^3_n)$ in $T_pX$, set $\mu'(x^1,x^2,x^3)=\lim_\omega\mu(x^1_n,x^2_n,x^3_n)$. The fact that $\mu$ is 1--Lipschitz in each factor implies that $\mu'$ is independent of the choice of representatives of the $x^i$, similarly to the proof of Lemma~\ref{lem:qi_induces_bilipschitz}. 

For each $i\ne j$ we have $\dist(x^i_n,x^j_n)=\dist(x^i_n,\mu(x^1_n,x^2_n,x^3_n))+\dist(\mu(x^1_n,x^2_n,x^3_n),x^j_n)$, so in the ultralimit we get $\dist(x^i,x^j)=\dist(x^i,\mu'(x^1,x^2,x^3))+\dist(\mu'(x^1,x^2,x^3),x^j)$. In particular, $\mu'$ produces a well-defined point of $T_pX$ satisfying the desired equalities. A simple computation shows that any point satisfying those equalities must actually be $\mu'(x^1,x^2,x^3)$, which shows that $(T_pX,\mu')$ is a median metric space.

Since $X$ is complete and connected, it is geodesic. As an ultralimit of a geodesic spaces, $T_pX$ is geodesic \cite[Prop.~3.4]{kapovichleeb:onasymptotic}. In particular, it is connected. It remains to bound~$\rk T_pX$. 

Suppose that there exists a median monomorphism $f:\{0,1\}^k\to T_pX$. Let us write $Q=\{0,1\}^k$, and $\mu_Q$ for its median operator. For each $q\in Q$, let $(x^q_n)$ be a sequence in $X$ representing $f(q)$. Since $f(Q)$ is a finite median subalgebra of $T_pX$, this must be captured by the approximating sequences. More concretely, for every $\eps>0$ we must have 
\[
M_\eps \,=\, \big\{m\in\mathbf N \,:\, 
    \lambda_m\dist\big(\mu(x^{q_1}_m,x^{q_2}_m,x^{q_3}_m),\,x^{\mu_Q(q_1,q_2,q_3)}_m\big)
    <\eps \text{ for all } q_1,q_2,q_3\in Q\big\} \,\in\,\omega.
\]
Let $r>0$ be such that $f(Q)$ is $r$--separated. For every $\eps>0$, we also have
\[
N_\eps \,=\, \{m\in\mathbf N \,:\, \{x^q_m\,:\,q\in Q\} \text{ is } 
    \frac{r-\eps}{\lambda_m}\text{-- separated}\} \,\in\,\omega.
\]
In particular, there exists $m\in N_{\frac r2}\cap M_{\frac r{2^{n+3}}}$, because $\omega$ is an ultrafilter. 

For $i\in\{1,\dots,k\}$, let $e_i$ denote the point in $Q$ with $i\th$ coordinate 1 and all other coordinates 0, and let $e_0$ denote the point $(0,\dots,0)$. For $i\in\{0,\dots,k\}$, let $B_i$ denote the median-convex hull of the ball $B_X(x^{e_i}_m,\frac r{10n\lambda_m})$. By Lemma~\ref{lem:ball_hull}, we have $B_i\subset B_X(x^{e_i}_m,\frac r{8\lambda_m})$. Since the set $\{x^q_m\,:\,q\in Q\}$ is $\frac r{2\lambda_m}$--separated, the $B_i$ are pairwise disjoint. By \cite[Thm~2.7]{roller:poc}, for each $i\in\{1,\dots,k\}$ there is a wall $h_i$ of $X$ separating $B_0$ from $B_i$. The $h_i$ must cross pairwise. Indeed, for $q\in Q$ the median $\mu(x^{e_0}_m,x^{e_i}_m,x^q_m)$ lies in $B_i$ if and only if the $i\th$ coordinate of $q$ is 1, so the crossing of the $h_i$ is witnessed by the set $\{x^q_m\,:\,q\in Q\}$. This shows that $k\le n$.
\end{proof}

In work on the asymptotic cones of mapping class groups \cite{behrstockminsky:dimension}, Behrstock--Minsky introduced a notion of dimension, later called \emph{separation dimension} by Bowditch \cite{bowditch:coarse}, that is a simple tweak on the more standard notion of \emph{inductive dimension} \cite{hurewiczwallman:dimension,engelking:theory}.

\begin{definition}[Separation dimension]
Let $Y$ be a Hausdorff topological space. The \emph{separation dimension} of $Y$ is defined inductively as follows. 
\begin{itemize}
\item   If $Y=\varnothing$, then $\sepdim Y=-1$.
\item   Otherwise, $\sepdim Y\le n$ if for each distinct $x,y\in Y$ there exist closed subsets $A,B\subset Y$ with $x\not\in B$, $y\not\in A$, and $Y=A\cup B$, such that $\sepdim(A\cap B)\le n-1$.
\end{itemize}
\end{definition}

Let $X$ be a metric space. By definition, $\sepdim X$ is always bounded above by the inductive dimension of $X$, which in turn is equal to the \emph{topological} (or \emph{covering}) dimension of $X$ by the Kat\v etov--Morita theorem \cite{katetov:ondimension,morita:normal} (see also \cite[Thm~4.1.3]{engelking:theory}). If $X$ is proper, then $\sepdim X$ is equal to the topological dimension \cite[\S III.6]{hurewiczwallman:dimension}. This is not true in general: the rational points of Hilbert space have separation dimension zero but topological dimension one \cite{erdos:dimension}. In our setting we have the following.

\begin{lemma}[{\cite[Cor.~3.7]{haettel:higher}}] \label{lem:sepdim_is_rank}
If $X$ is a connected median metric space, then $\sepdim X=\rk X$. 
\end{lemma}

\subsection{Coarse medians}

\begin{definition}[Quasimedian map]
Let $X$ and $Y$ be metric spaces equipped with ternary operators $\mu_X$ and $\mu_Y$, respectively. A map $f:X\to Y$ is said to be \emph{$q$--quasimedian} if $\dist_Y(f\mu_X(x,y,z),\mu_Y(fx,fy,fz))\le q$ for all $x,y,z\in X$.
\end{definition}

The following definition can be thought of as a higher-rank version of Gromov's tree approximation lemma for hyperbolic spaces \cite{gromov:hyperbolic}.

\begin{definition}[Coarse median space]
Let $X$ be a metric space. A \emph{coarse median} on $X$ is a ternary operator $\mu:X\to X$ such that there is some sequence $(h_n)$ with the following properties.
\begin{itemize}
\item   $\mu$ is $h_0$--coarsely Lipschitz in each factor.
\item   For each finite subset $A\subset X$ there is a finite median algebra $M$ with an $h_{|A|}$--quasimedian map $\iota:M\to X$ and a map $o:A\to M$ such that $\dist(\iota o(a),a)\le h_{|A|}$ for all $a\in A$.
\end{itemize}
We call $(X,\mu)$ a \emph{coarse median space}. If every $M$ can be chosen to have rank at most $n$, then we say $\mu$ has rank at most $n$, writing $\rk\mu\le n$. We write $\rk X$ for the infimal rank of coarse medians on $X$.
\end{definition}

If a metric space $X$ admits a coarse median, then we shall often simply refer to $X$ as a coarse median space. The definition of a coarse median space can also be formulated in terms more similar to Definition~\ref{def:median} \cite{niblowrightzhang:four}. The following lemma provides another link with median metric spaces.

\begin{lemma}[{\cite[Thm~6.9]{bowditch:large:mapping}}] \label{lem:cone_median}
Let $(X,\mu)$ be a coarse median space with $\rk\mu\le n$, and let $\hat X$ be an asymptotic cone of $X$. After a bilipschitz change of metric, $(\hat X,\hat\mu)$ is a median metric space of rank at most $n$.
\end{lemma}

Whilst the fact that $\hat X$ is a median metric space of finite rank is already useful, the exact control on ranks will be refined in Proposition~\ref{prop:rank} below, which shows that if $\hat X$ is an asymptotic cone of a proper cocompact coarse median space $X$, then $\rk X=\rk\hat X$. This justifies the similarity in notation between the rank of a coarse median space $X$ and that of a median algebra.

A subset of a coarse median space $(X,\mu)$ is a \emph{quasisubalgebra} if it is the image of some median algebra $M$ under a quasimedian map $M\to X$. A subset $A\subset X$ is \emph{$k$--coarsely convex} if $\mu(a,a',x)$ is $k$--close to $A$ for all $a,a'\in A$, $x\in X$. As in the setting of median algebras, for $A\subset X$ let $J(A)=\{\mu(a,a',x)\,:\,a,a'\in A,\, x\in X\}$. Following \cite{bowditch:convex}, and in analogy with the setting of median algebras, if $\rk\mu=n$, then the \emph{coarse-median hull} of a subset $A\subset X$ is defined to be $J^n(A)$. It can be checked that coarse-median hulls are uniformly coarsely convex.

\subsection{Admissible graphs of groups}

As part of the applications considered in Sections~\ref{sec:fbz} and~\ref{sec:tubular} we shall consider graphs of groups. Since the general theory is somewhat standard, we refer the reader to \cite{scottwall:topologicalmethods} for a full discussion. The proof of Proposition~\ref{prop:few_linear} uses the \emph{admissible} graph of groups defined in \cite{crokekleiner:geodesic}. Here we introduce enough notation to state the definition.

For an edge $e$ in a graph, write $e^-$ and $e^+$ for its two incident vertices. In brief, a graph of groups $\cal G$ consists of: a nontrivial, finite, connected graph $\Lambda$; a \emph{vertex group} $G_v$ for each vertex $v\in\Lambda$; an \emph{edge group} $G_e$ for each edge $e\in\Lambda$; and injective homomorphisms $G_e\to G_{e^\pm}$ so that $G_e$ is identified with subgroups of $G_{e^\pm}$. This collection of data can be used to define a group $G$, and one calls $\cal G$ a \emph{graph of groups decomposition} of $G$.

\begin{definition}[Admissible] \label{def:admissible}
A graph of groups $\cal G$ is \emph{admissible} if the following hold.
\begin{itemize}
\item   Each vertex group $G_v$ has centre $Z_v\cong\Z$, and ${}^{G_v}\!/_{Z_v}$ is non-elementary hyperbolic.
\item   Each edge group is isomorphic to $\Z^2$.
\item   For each edge $e$, the subgroup $\sgen{G_e\cap Z_{e^-}, G_e\cap Z_{e^+}}$ has finite index in $G_e$.
\item   For each vertex group $G_v$, distinct conjugates in $G_v$ of any two (possibly equal) incident edge groups are non-commensurable.
\end{itemize}
\end{definition}

\section{Dimension bounds} \label{sec:dimension}

The goal of this section is to upper bound the minimal rank of a coarse median space in terms of its geometry. By the \emph{quasiflat rank} of a metric space $X$, we mean the supremal integer $\qfrk X$ for which there is a quasiisometric embedding $\R^{\qfrk X}\to X$.

\begin{proposition} \label{prop:rank}
Let $X$ be a coarsely connected coarse median space, and let $\hat X$ be an asymptotic cone of $X$; it is automatically a connected median algebra of rank $\rk\hat X$. We have $\sepdim\hat X=\rk\hat X\le\rk X$. If $X$ is proper and has cocompact isometry group, then the following quantities agree.
\begin{itemize}
\item   $\rk X$.
\item   $\rk\hat X$.
\item   $\sepdim\hat X$.
\item   $\qfrk X$
\item   The supremal $d$ such that there is a quasiisometric embedding $\R^d\to X$ that is quasimedian for some coarse median on $X$ realising $\rk X$.
\end{itemize} 
\end{proposition}

\begin{proof}
The agreement of $\sepdim\hat X$ with $\rk\hat X$ is given by Lemmas~\ref{lem:sepdim_is_rank} and~\ref{lem:cone_median}. In particular, since separation dimension is a topological property, $\rk\hat X$ is independent of the choice of coarse median on $X$. According to \cite[Thm~2.3]{bowditch:coarse}, we have $\rk X\ge\rk\hat X$. Also, if $\qfrk X\ge d$, then $\hat X$ contains a bilipschitz copy of $\R^d$, and hence $\sepdim\hat X\ge\qfrk X$. 

Now suppose that $X$ is proper and has cocompact isometry group. We shall construct a quasiflat of dimension $\rk X$, which will show the equivalence between the first four bulleted items.

Let $d=\rk X$, and let $\mu$ be a coarse median on $X$ with $\rk\mu=d$. Let $(h^\mu_n)$ be a corresponding sequence. For any increasing sequence $(h_n)$ there is a finite set $A\subset X$ such that there is no median algebra $M$ of rank $d-1$ admitting a $3h_{3^{2d}|A|}$--quasimedian map $\iota:M\to X$ and a map $o:A\to M$ such that $\dist(\iota o(a),a)\le 3h_{3^{2d}|A|}$ for all $a\in A$. We are free to assume that $h_n\ge 3h^\mu_n$ for all $n$.

Let $N$ be a median algebra of rank $d$ that approximates $A$ with respect to $\mu$. According to \cite[Prop.~8.2.4]{bowditch:median:book}, in any median algebra of rank $d$, the subalgebra generated by a subset is generated by taking medians at most $2d$ times. Hence $|N|\le3^{2d}|A|$. Let $\iota^\mu$ be the map corresponding to $\mu$, $A$, and $N$. The set $A$ lies in the $h^\mu_{|A|}$--neighbourhood of $\iota^\mu N$, so the latter, whose cardinality is at most $3^{2d}|A|$, cannot be approximated by any median algebra of rank $d-1$ with error at most $2h_{3^{2d}|A|}$. As $\iota^\mu N$ is an $h^\mu_{|A|}$--quasisubalgebra, this shows that there is an $h^\mu_{|A|}$--quasimedian embedding of a cube $\{0,1\}^d\to X$ whose image is $h_{3^{2d}|A|}$--separated. Approximating the images of those cubes with $\mu$ and varying the sequence $(h_n)$, we conclude that $X$ contains arbitrarily large $h^\mu_{2^d}$--quasicubes of dimension $d$.

Fix a basepoint $x_0\in X$. For each $n$, let $Q_n$ be a uniform quasicube in $X$ of dimension $d$ whose vertices are $n$--separated. By an application of \cite[Lem.~9.1, Prop.~9.3]{bowditch:quasiflats}, for sufficiently large $n$, we can take $Q_n$ to be the image of a product of $d$ real intervals of length $n$ under a uniform-quality quasimedian quasiisometric embedding. Write $z_n$ for the central point of $Q_n$. By cocompactness, we can translate $z_n$ into a fixed compact set $C$ containing $x_0$ by an isometry $g_n$. This gives uniform quasicubes in the sequence of coarse median spaces $(X,g_n\mu)$ that are all centred in the compact set $C$.

Now take an (unrescaled) ultralimit. Since $X$ is proper, we have $X=\lim_\omega B(x_0,m)$, where $B(x_0,m)$ is the ball of radius $m$ centred on $x_0$. Choosing $m_n$ so that $g_nQ_n\subset B(x_0,m_n)$, we get that $Q=\lim_\omega g_nQ_n\subset X$, because $g_nQ_n$ is centred in the fixed compact set $C$. Because the $Q_n$ are uniform quasiisometric embeddings of increasingly large $d$--cubes, $Q$ is a quasiisometric embedding of $\R^d$. We have found the desired $d$--quasiflat in $X$.

It remains to show that there is a coarse median realising $\rk X$ for which the embedding of $Q$ is quasimedian. It is easy to see that the conditions defining a coarse median of rank $\rk X$ hold for the ultralimit $\lim_\omega g_n\mu$, and $Q\to X$ is quasimedian with respect to this because $g_nQ_n\to X$ is uniformly quasimedian with respect to $g_n\mu$.
\end{proof}


The conclusion of Proposition~\ref{prop:rank} can fail without the cocompactness assumption, as shown by the following. 

\bsh{Example} \label{eg:log}
Let $X\subset(\R^2,\ell^1)$ be bounded between the $x$--axis, the line $x=1$, and the graph of the function $x\mapsto\log x$. Since $X$ is affinely convex in the plane, it is a median subalgebra. As a coarse median space, $\rk X=2$, because it contains arbitrarily large squares. However, $\hat X$ is a ray.

Let $Y$ be the median metric space constructed from the real line by attaching, for each $n\in\mathbf N$, a square of side-length $n$ at the point $n\in\R$, along a vertex. Clearly $\qfrk Y=1$, but both $Y$ and $\hat Y$ contain arbitrarily large squares.
\esh

\begin{remark}[Asymptotic rank]
The behaviour in Example~\ref{eg:log} arises because asymptotic cones are taken with a fixed basepoint and the space is not homogeneous. However, there is still something to be said even without the assumptions of properness and cocompactness. Following \cite{wenger:asymptotic} (see Proposition~3.1 thereof), the \emph{asymptotic rank} of a metric space $X$ can be defined to be the supremal $n$ such that there is some asymptotic cone (with basepoints allowed to move) $\hat X$ of $X$ and a sequence of subspaces $B_k\subset X$ whose limit in $\hat X$ is the unit ball in some normed space $(\R^n,\|\cdot\|)$.

The same proof as in Proposition~\ref{prop:rank} shows that if $X$ is a coarsely connected coarse median space of rank $d$, then the asymptotic rank of $X$ is $d$. Indeed, consider the subset $Q_n$ constructed in the proof of Proposition~\ref{prop:rank}, which is the image of a product of $d$ real intervals of length $n$ under a uniform-quality quasimedian quasiisometric embedding. Let $z_n$ be the centre of $Q_n$. The limit of the sequence $(Q_n)$ in the asymptotic cone $\hat X=\lim_\omega(X,\frac1{\sqrt n},z_n)$ is $(\R^d,\ell^1)$, which implies the existence of the desired sequence. The cocompactness was only used to control the basepoint, and the properness was only used to say that the space was equal to its own ultralimit.
\end{remark}


Whilst Proposition~\ref{prop:rank} is very precise, it is not always easy to ascertain the quasiflat rank of a given group $G$. In practice, therefore, it is useful to have a statement in terms of a more easily calculable quantity. The following is Theorem~\ref{mthm:vgd} from the introduction.

\begin{corollary} \label{thm:dimension}
Let $G$ be a finitely generated group. If $G$ admits a coarse median, then $\rk G\le\vcd G$. 
\end{corollary}

\begin{proof}
According to \cite[Thm~1.2]{sauer:homological}, any group that contains a quasiflat of dimension $n$ must have cohomological dimension at least $n$. Since Proposition~\ref{prop:rank} shows that $\rk G=\qfrk G$, this proves the result.
\end{proof}

By the Eilenberg--Ganea theorem, if a finitely presented group $G$ has cohomological dimension not equal to two, then its cohomological dimension is equal to its \emph{geometric dimension}: the infimal dimension of a $K(G,1)$. Moreover, it remains unknown whether the same holds for groups of cohomological dimension two. Thus, in almost all cases Corollary~\ref{thm:dimension} can be equivalently stated using virtual geometric dimension. In this case the application of Sauer's theorem, whose statement is considerably more general than is being used here, can be replaced by Proposition~\ref{lem:qfrk_vgd} below, which has a comparatively short proof.

\begin{definition}
Let $X$ be a combinatorial cell complex, and let $c\in C_n(X)$ be an $n$--chain. Letting $c=\sum_{\sigma\in\supp c}a_\sigma\sigma$, we write $|c|=\sum_{\sigma\in\supp c}|a_\sigma|$. A function $f$ is called a \emph{$k^\mathrm{th}$--order homological isoperimetric function} for $X$ if for each $k$--boundary $b$ there is a $(k+1)$--chain~$c$ with $\partial c=b$ and $|c|\le f(|b|)$. The \emph{$k^\mathrm{th}$--order homological Dehn function} of $X$ is the minimal $k^\mathrm{th}$--order homological isoperimetric function.
\end{definition}

Note that we are not considering these functions up to the usual equivalence. We are interested in slightly more precise control for specific complexes. The following can be extracted from the proof of \cite[Thm~2.1]{fletcher:homological}.

\begin{lemma} \label{lem:homological_filling}
For each $k,q$ there exists $C$ such that the following holds. Let $Y'$ and $Z'$ be connected combinatorial cell complexes with finite $(k+1)$--skeletons, and let $Y$ and $Z$ be their universal covers. Suppose that $Y$ and $Z$ are $k$--connected and let $D_{k,Y}$ be the $k^\mathrm{th}$--order homological Dehn function of $Y$. If $f:Y\to Z$ is a $q$--quasiisometric embedding, then for every $k$--boundary $b$ in $f(Y)$, if $c$ is a $(k+1)$--chain in $Z$ with $\partial c=b$, then $|c|\ge\frac1CD_{k,Y}(\frac{|b|}C)-C|b|$.
\end{lemma}

For a finitely generated group $G$, write $\gd G$ for the geometric dimension of $G$. Recall that $G$ has \emph{type $F$} it has a finite $K(G,1)$, and \emph{type $F_\infty$} if it has a $K(G,1)$ whose $n$--skeleton is finite for all $n$.

\begin{proposition} \label{lem:qfrk_vgd}
If $G$ is a group of type $F_\infty$, then $\qfrk G\le\gd G$.
\end{proposition}

\begin{proof}
Suppose that $\gd G<\infty$. Whilst it remains open whether $G$ must be of type $F$, a combination of Propositions~7.2.13 and~7.2.15 of \cite{geoghegan:topological} shows that $G\times\Z$ is of type $F$. It follows from \cite[VIII.7.1]{brown:cohomology} that we can find a finite $K(G\times\Z,1)$ of dimension $d=1+\gd G$. Let $X$ be its universal cover.

If $\qfrk G>\gd G$, then clearly $\qfrk(G\times\Z)\ge d+1$. Since $X$ is quasiisometric to $G\times\Z$, there is a $q$--quasiisometric embedding $f:\R^{d+1}\to X$ for some $q$. Let $S_n$ be the sphere of radius $n$ in $\R^{d+1}$ centred at the origin. Let $E_n$ be its equator: the intersection of $S_n$ with the hyperplane $\{(z_0,\dots,z_d)\in\R^{d+1}\,:\,z_0=0\}$. Let $H_n^+$ and $H_n^-$ be the two hemispheres of $S_n$ that meet in $E_n$.

Up to a uniformly bounded perturbation, simplicial approximation implies that $b_n=f(E_n)$ is a $(d-1)$--cycle in $f(\R^{d+1})$. As $X$ is contractible, $b_n$ is a $(d-1)$--boundary. By Lemma~\ref{lem:homological_filling}, there is a constant $C=C(d,q)$ such that for every $n$, every $d$--chain $c_n$ in $X$ with $\partial c_n=b_n$ has $|c_n|\ge\frac1CD_{d-1}\big(\frac{|b_n|}C\big)-C|b_n|$, where $D_{d-1}$ is the $(d-1)^\mathrm{th}$--order homological Dehn function of $\R^{d+1}$. In particular, $|c_n|$ is bounded below by a fixed superlinear function of $|b_n|$.

By the construction of $b_n$, up to a small perturbation it is filled by $f(H_n^+)$. Let us write $c^+_n$ for this filling. Since $f$ is a quasiisometric embedding, there is a divergent function $\delta:\mathbf N\to \R_{>0}$ such that $\supp c^+_n$ contains a $d$--cell $z_n\subset X$ at a distance of at least $\delta(n)$ from $b_n$. We can also fill $b_n$ with (a perturbation of) $f(H_n^-)$. Let us write $c^-_n$ for this filling. 


The $d$--chain $c^+_n\cup c^-_n$ has zero boundary, so since $H_d(X,\Z)=0$ and $X$ has no $(d+1)$--cells, the coefficient of every $d$--cell must be zero. In particular, $z_n\in\supp c^-_n$. Since the distance from $z_n$ to $c^-_n$ diverges, this eventually contradicts the assumption that $f$ is a quasiisometric embedding. 
\end{proof}


We conclude by proving a geometric-dimension variation of Corollary~\ref{thm:dimension}.

\begin{theorem} 
Let $G$ be a finitely generated group. If $G$ admits a coarse median, then $\rk G\le\vgd G$. 
\end{theorem}

\begin{proof}
By Proposition~\ref{prop:rank}, we have $\rk G=\qfrk G$. According to \cite[Prop.~12.4.7]{bowditch:median:book}, every asymptotic cone of $G$ is $n$--connected for all $n$. By \cite[Thm~D]{riley:higher}, we find that $G$ is of type $F_\infty$. Suppose that $\vgd G<\infty$, and let $H$ be a finite-index subgroup of $G$ realising $\vgd G$. It is of type $F_\infty$, for instance by \cite[IX.6.1]{brown:cohomology}. The result follows from Proposition~\ref{lem:qfrk_vgd}, because $H$ is quasiisometric to $G$.
\end{proof}

\section{Richly branching flats} \label{sec:rbf}

Here we describe certain geometric configurations whose appearance in a metric space $X$ prevents $X$ from admitting a coarse median of low rank. The functional root of these obstructions is the following.

\begin{lemma} \label{lem:diagonal_five_point}
Let $n\ge2$ and equip $(\R^n,\ell^1)$ with its standard median. Let $v\in\R^n$. Let $H^+$ and $H^-$ be the two halfspaces of $\R^n$ bounded by $v^\bot$. Let $X$ be obtained from $\R^n$ by gluing a copy $I$ of $[0,\infty)\times\R^{n-1}$ to $\R^n$ along $v^\bot$. If $X$ is a median metric space such that $I\cup H^+$ is median isometric to $(\R^n,\ell^1)$, then $v$ is parallel to some coordinate axis.
\end{lemma}

\begin{proof}
Suppose that $I\cup H^+$ is median isometric to $\R^n$, but that $v$ is not parallel to any coordinate axis. Let $\bf0$ denote the origin of $\R^n$. There is some point $\bf1\in v^\bot$ whose coordinates are all nonzero. Inside $\R^n$, the median interval from $\bf0$ to $\bf1$ is an $n$--box, i.e. a product of $n$ nontrivial intervals. Let $a^+$ be one of its vertices in $H^+$, and let $a^-$ be the opposite vertex, which lies in $H^-$.

Since $I\cup H^+$ is median isometric to $\R^n$, the median interval in $I\cup H^+$ from $\bf0$ to $\bf1$ is also an $n$--box, with $a^+$ as one of its vertices. Let $b$ be the vertex opposite $a^+$, which lies in $I$. We have the following identities.
\[
\mu(\bf0,\bf1,a^\pm)=a^\pm, \quad \mu(a^+,a^-,\bf0)=\bf0, \quad \mu(a^+,a^-,\bf1)=\bf1,
\]
\[
\mu(\bf0,\bf1,b)=b, \quad \mu(a^+,b,\bf0)=\bf0, \quad \mu(a^+,b,\bf1)=\bf1.
\]
By repeatedly applying the five-point condition and these identities, we can now make the following computation.
\begin{align*}
b \;&=\; \mu(\bf0,\bf1,b), \;=\; \mu\big(\mu(a^+,a^-,\bf0),\mu(a^+,a^-,\bf1),b\big) 
    \;=\; \mu\big(a^+,a^-,\mu(\bf0,\bf1,b)\big) \\
&=\; \mu\big(\mu(a^+,a^-,b),\mu(a^+,a^-,\bf0),\bf1\big) \;=\; \mu\big(\mu(a^+,a^-,b),\bf0,\bf1\big) \\
&=\; \mu\big(\mu(a^+,b,a^-),\mu(a^+,b,\bf0),\bf1\big) \;=\; \mu\big(a^+,b,\mu(a^-,\bf0,\bf1)\big) \\
&=\; \mu\big(\mu(a^+,b,\bf0),\mu(a^+,b\bf1),a^-\big) \;=\; \mu(\bf0,\bf1,a^-) \;=\; a^-.
\end{align*}
This contradiction shows that $v$ must be parallel to some coordinate axis.
\end{proof}

The idea for the configurations we shall consider is that they contain enough branching to force the condition of Lemma~\ref{lem:diagonal_five_point} to fail when one passes to the asymptotic cone. Indeed, there are only $n$ coordinate axes in $\R^n$, so one would expect $n+1$ directions of branching to be enough. One needs to be a little more careful in order to get coarse obstructions, because quasiflats only yield bilipschitz flats in the asymptotic cone. For instance, the cyclic union of six quarterplanes is a median metric space bilipschitz to $\R^2$, and there can be branching along three lines through the origin.

\begin{definition}[RBF] \label{def:rbf}
For a natural number $n\ge2$, an \emph{$n$--dimensional richly branching flat}, or \emph{$n$--RBF}, is a piecewise linear space $R$ constructed as follows. Let $B$, the \emph{base flat}, be an isometric copy of $\R^n$. Let $v_0,\dots,v_n$ be pairwise linearly independent vectors in $B$. For each $i$, choose a coarsely dense subset $P_i\subset\R$. To obtain $R$ from $B$, glue, along its boundary, a copy of the half-flat $\R^{n-1}\times[0,\infty)$ along each codimension-1 affine subspace of the form $pv_i+v_i^\bot$ with $p\in P_i$. See Figure~\ref{fig:RBF}.
\end{definition}

\begin{theorem} \label{thm:rbf}
Let $X$ be a coarse median space with $\rk X\le n$. There is no quasiisometric embedding of an $n$--RBF into $X$.
\end{theorem}

\begin{proof}
Suppose that $f:R\to X$ is a quasiisometric embedding in $X$ of an $n$--RBF. Let $\hat R$ be the asymptotic cone of $R$ with respect to some ultrafilter $\omega$, some sequence $(\lambda_m)$, and some basepoint $b$. The base flat $B\subset R$ yields a subspace $\hat B\subset\hat R$ that is isometric to $\R^n$. Let $\hat X$ be the asymptotic cone of $X$ with respect to $\omega$, $(\lambda_m)$, and $f(b)$. Lemma~\ref{lem:qi_induces_bilipschitz} shows that $f$ induces a bilipschitz embedding $\hat f:\hat R\to\hat X$. In particular, $\hat X$ contains bilipschitz copies of $\R^n$, and so has separation dimension at least $n$. By Proposition~\ref{prop:rank}, we must have $\rk X=\rk\hat X=n$.

According to \cite[Lem.~5.2]{bowditch:quasiflats}, $\hat f\hat B$ is a finite union of isometric, median embedded panels.
There must be a panel $P$ of $\hat f\hat B$ that is the image of an isometric, median embedding $\phi:[0,\infty)^n\to\hat f\hat B\subset\hat X$. Composing gives a bilipschitz embedding $g=\hat f^{-1}\phi:[0,\infty)^n\to\hat B$. Let $U\subset\hat B$ be an open ball in the image of $g$. By Rademacher's theorem, both $g$ and $g^{-1}$ are almost everywhere differentiable on their domains. Hence there is a point $p\in U$ in the image of $g$ such that $g^{-1}$ is differentiable at $p$ and $g$ is differentiable at $g^{-1}(p)$.

For small $\eps>0$, every line segment $\gamma_i:t\mapsto p+tv_i$ defined on $(-\eps,\eps)$ is contained in $U$. The $v_i$ are pairwise independent, so the fact that $g^{-1}$ is bilipschitz means that the tangent vectors $\frac d{dt}(g^{-1}\gamma_i)(0)$ must be pairwise independent as well. Consequently, there must be some $i$ such that $\frac d{dt}(g^{-1}\gamma_i)(0)$ is not parallel to any coordinate axis of $[0,\infty)^n$.

Let $F\subset\hat R$ be an $n$--flat consisting of a half-flat in $\hat B$ glued to a half-flat meeting $\hat B$ along $p+v_i^\bot$. By a similar argument to the above, we can make a small perturbation of $p$ within $U$ so that 
\begin{itemize}
\item 	$\hat f(p)$ lies in the interior of $P$ and in the interior of a panel of $\hat f F$;
\item 	$g^{-1}$ is differentiable at $p$, and $g$ is differentiable at $g^{-1}(p)$;
\item 	$\frac d{dt}(g^{-1}\gamma_i)(0)$ is not parallel to any coordinate axis;
\item 	the analogous maps for $F$ are differentiable at $p$ and the respective preimage.
\end{itemize}

Consider the tangent cone $T_{\hat f(p)}\hat X$. According to Lemma~\ref{lem:tangent_median}, $T_{\hat f(p)}\hat X$ is a median metric space of rank at most $n$. The panel $P$ in which $\hat f(p)$ lives induces a flat subspace $B'$ of $T_{\hat f(p)}\hat X$ that is median isometric to $(\R^n,\ell^1)$. Because $\hat f(p)$ is in the interior of a panel of $\hat fF$ and $F$ is glued to $\hat B$ along $p+v_i^\bot$, the above property of derivatives produces a half-flat $I$ glued to $B'$ along a codimension-1 subspace whose normal vector is not parallel to any coordinate axis. Moreover, the union of $I$ with a halfspace of $B'$ is median isometric to $(\R^n,\ell^1)$. But this contradicts Lemma~\ref{lem:diagonal_five_point}.
\end{proof}

\section{Free-by-cyclic groups} \label{sec:fbz}

\emph{All free groups considered in this section will be finitely generated}. Let $F$ be a (finitely generated) free group, and let $\phi\in\Aut F$. The \emph{free-by-cyclic} group corresponding to $\phi$ is the group $G=F\rtimes_\phi\Z$. For each $k>0$, the group $F\rtimes_{\phi^k}\Z$ is a finite-index subgroup of $G$.

Each $\phi\in \Aut F$ is induced by some homotopy equivalence $f$ of a graph $\Gamma$ with $\pi_1\Gamma=F$. An important perspective is to view $G$ as the fundamental group of the corresponding mapping torus. This is especially effective when $f$ has the structure of a \emph{train track}. There are several versions of train tracks, of varying levels of technicality and strength \cite{bestvinahandel:train,bestvinafeighnhandel:tits:1,feighnhandel:recognition}. Here we shall make use of the \emph{improved relative train track maps} of \cite[Thm~5.1.5]{bestvinafeighnhandel:tits:1}. Below we summarise the parts of this machinery that are needed for our application.

\begin{definition}[Nielsen path]
Let $\Gamma$ be a graph, and let $f:\Gamma\to\Gamma$ be a homotopy equivalence. For a path $\gamma\subset\Gamma$, let $f_\#(\gamma)$ be the unique immersed path homotopic to $f\gamma$ and with the same endpoints. We say that $\gamma$ is a \emph{Nielsen path} if $f_\#(\gamma)=\gamma$. By a \emph{Nielsen cycle}, we mean a nontrivial, immersed cycle $\gamma$ so that $f_\#(\gamma)=\gamma$. 
\end{definition}

Recall that a \emph{filtration} of a graph $\Gamma$ is a sequence of nested subgraphs $\varnothing=\Gamma_0\subset\Gamma_1\subset\dots\subset\Gamma_n=\Gamma$. The \emph{strata} of the filtration are the subgraphs $\Gamma_i\ssm\Gamma_{i-1}$. Note that the $\Gamma_i$ are not necessarily connected. 

\bsh{Improved relative train tracks}[{\cite[Thm~5.1.5]{bestvinafeighnhandel:tits:1}}] \label{sh:irtt}
Let $\phi$ be an automorphism of a free group $F$. After replacing $\phi$ by some positive power, there is a connected finite graph $\Gamma$ with $\pi_1\Gamma=F$, a filtration $\varnothing=\Gamma_0\subset\Gamma_1\subset\dots\subset\Gamma_n$, and a particularly nice homotopy equivalence $f:\Gamma\to\Gamma$ inducing $\phi$. More specifically, $f$ can be chosen to have the properties described below.

\begin{itemize}[leftmargin=.8cm]
\item   An edge of $\Gamma$ is \emph{invariant} if it is fixed by $f$. Let $k$ be the number of invariant edges. If $i\le k$, then the $i\th$ stratum consists of a single invariant edge.

\item   Strata after the $k\th$ are either \emph{exponential}, \emph{non-exponential}, or \emph{zero} strata. 

\item   Each non-exponential stratum consists of a single edge. Both vertices of each non-exponential stratum are fixed by $f$.

\item   If $e_i$ is a non-exponential stratum, then $f_\#(e_i)=e_iu_i$ for some closed path $u_i\subset\Gamma_{i-1}$, called the \emph{suffix} of $e_i$. If $u_i$ is a Nielsen path, then $f_\#$--iterates of $e_i$ grow in length linearly, and $e_i$ is called a \emph{linear} stratum.

\item   If there are non-exponential strata, then there must be linear strata.

\item   If $e_i$ is a linear stratum, then $u_i$ is a Nielsen \emph{cycle}. We say that a Nielsen cycle $u$ \emph{supports} the linear stratum $e_i$ if $u_i$ is a nonzero power of a cyclic permutation of~$u$. 

\item   Distinct linear edges have distinct suffixes \cite[Rem~3.12]{bestvinafeighnhandel:tits:2}.
\end{itemize}
The map $f$ is called an \emph{improved relative train track map}, and we refer to the entire package of data above as the \emph{IRTT structure} of $F\rtimes_\phi\sgen t$, or simply of $\phi$.
\esh

\bsh{Example} \label{eg:Hyp_rel_gersten}
The following is a useful example to bear in mind; we thank Naomi Andrew for suggesting it. Let $F=\sgen{a,b,c,d}$, and consider the automorphism
\[
\phi \,=\,  \begin{cases}   a\mapsto ab, & c\mapsto c[a,b], \\ b\mapsto bab, & d\mapsto d[a,b]^2. \end{cases}
\]
By \cite[Thm~3.11]{ghosh:relative} or \cite[Thm~4]{dahmanili:relative}, the group $G=F\rtimes_\phi\Z$ is hyperbolic relative to Gersten's group \cite{gersten:automorphism}.

An IRTT structure for $\phi$ is given as follows. Let $\Gamma_1$ be a rose on two petals labelled $a$ and~$b$, let $\Gamma_2$ be a rose on petals $a,b,c$, and let $\Gamma_3=\Gamma$ be a rose on petals $a,b,c,d$. The first stratum is exponential. The commutator $[a,b]$ represents a Nielsen cycle, so $c$ and $d$ are linear strata whose suffixes are powers of it. That is, $[a,b]$ supports $c$ and $d$.

Note that no cyclic free factors of $F$ are fixed by any power of $\phi$, so for any IRTT structure on a power of $\phi$ the invariant strata can only form a subforest. 
\esh

The following result shows that if there are few non-exponential strata then $G$ has fairly strong hyperbolic- and cubical-like features.

\begin{proposition} \label{prop:few_linear}
Let $G=F\rtimes_\phi\Z$, and fix an IRTT structure for (a power of) $\phi$.
\begin{itemize}
\item   If there are no Nielsen cycles, then $G$ is hyperbolic and cocompactly cubulated.
\item   If there are no linear strata, then $G$ is virtually hyperbolic relative to groups of the form $F'\times\Z$, where $F'$ is free.
\item   If all non-exponential strata are linear and each Nielsen cycle supports at most one linear stratum, then $G$ is virtually a \emph{colourable hierarchically hyperbolic group}.
\end{itemize}
In all three cases, $G$ is quasiisometric to a finite-dimensional CAT(0) cube complex.
\end{proposition}

\begin{proof}
There being no Nielsen cycles is equivalent to $\phi$ being \emph{atoroidal}, so Brinkmann's theorem states that $G$ is hyperbolic in this case \cite{brinkmann:hyperbolic}. It was proved by Hagen--Wise that hyperbolic free-by-cyclic groups are cocompactly cubulated \cite{hagenwise:cubulating:irreducible,hagenwise:cubulating:general}.

\mk

More generally, $G$ is (virtually) hyperbolic relative to its maximal polynomially-growing sub--mapping-tori \cite[Thm~4]{dahmanili:relative}, \cite[Thm~3.11]{ghosh:relative}. We claim that if there are no linear strata, then all such mapping tori are trivial bundles. Let $\gamma\subset\Gamma$ be an edge path such that $f_\#$--iterates of $\gamma$ grow in length polynomially. By a repeated application of \cite[Lem.~6.5]{levitt:counting}, we can split some iterate $f^m_\#\gamma$ as $f^m_\#\gamma=\gamma_1\dots\gamma_n$ in such a way that there is no cancellation between $f^r_\#\gamma_i$ and $f^r_\#\gamma_{i+1}$ for any $i$ or $r$, and such that each $\gamma_i$ is one of: an edge of $\Gamma$; a Nielsen path; or an \emph{exceptional path}. An exceptional path is a path of the form $e_1p^ke_2^{-1}$, where $e_1$ and $e_2$ are (possibly equal) linear strata whose suffixes are both powers of $p$.

Since $f_\#$--iterates of $\gamma$ grow polynomially, no $\gamma_i$ can have exponentially growing iterates. Since there are no linear strata, every $\gamma_i$ that is an edge of $\Gamma$ is an invariant edge. Moreover, there can be no exceptional paths in $\Gamma$. Thus each $\gamma_i$ is either an invariant edge or a Nielsen path. Since there is no cancellation between iterates of the $\gamma_i$, this shows that $\gamma$ is a Nielsen path. Hence all polynomially-growing sub--mapping-tori of $G$ are trivial bundles, so $G$ is hyperbolic relative to groups of the form $F'\times\Z$ where $F'$ is free, as claimed. We observe for later that \cite[Thm~9.1]{behrstockhagensisto:hierarchically:2} implies $G$ is a colourable hierarchically hyperbolic group.

\mk

Finally, suppose that all non-exponential strata are linear, but that each Nielsen cycle supports at most one linear stratum. Since the property of being a colourable hierarchically hyperbolic group is preserved by relative hyperbolicity \cite[Thm~9.1]{behrstockhagensisto:hierarchically:2}, we can use \cite[Thm~4]{dahmanili:relative} or \cite[Thm~3.11]{ghosh:relative} to assume that $\phi$ is linearly growing. There is necessarily at least one linear stratum. 

We start by describing a graph-of-groups decomposition of $G$, following the discussion of \cite[Prop.~5.2.2]{andrewmartino:free-by-cyclic}. The parabolic orbits theorem \cite[Thm~13.2]{cohenlustig:very} (see also \cite[Thm~2.4.9]{andrewmartino:free-by-cyclic}) states that there is a unique simplicial $F$--tree preserved by (a power of) $\phi$. This gives a splitting of $F$. The vertex groups in this splitting are free of rank at least two. The edge groups are exactly the cyclic subgroups generated by Nielsen cycles supporting linear strata (see \cite[Lem.~11.5]{cohenlustig:very}, \cite[Def.~4.36]{bestvinafeighnhandel:tits:2}), which are maximal cyclic subgroups. (Strictly speaking, in order to correctly respect basepoints, they are generated by conjugates of such Nielsen cycles by paths in a fixed spanning tree of $\Gamma$.) The edge inclusions carry the information of the suffixes of these linear strata. By the assumption that no Nielsen cycle supports more than one linear stratum, no two adjacent edge groups can be equal.

Taking the $\phi$--mapping-torus of this splitting of $F$ yields a splitting of (a finite-index subgroup of) $G$. The fact that there is at least one linear stratum implies that there is at least one edge in this splitting. Each edge group is a maximal $\Z^2$ subgroup, namely the mapping torus of a Nielsen cycle in the splitting of $F$, with the inclusion maps corresponding to pulling its basepoint across the edge. The vertex groups are maximal $F'\times\Z$ subgroups; see \cite{andrewmartino:free-by-cyclic}, and also relevant discussion in \cite[\S2.4]{dahmanitouikan:unipotent}. 

We show that this graph of groups $\cal G$ is admissible (Definition~\ref{def:admissible}). The fact that each Nielsen cycle supports at most one linear stratum directly implies that the inclusions of distinct edge groups are not commensurable: they represent independent geodesics in the vertex group of the splitting of $F$. For similar reasons, the maximality of the edge groups ensures that the image of an edge group is not commensurable with any of its conjugates. Lastly, let $E$ be an edge group. The centres of the incident vertex groups correspond to the fibres of the respective mapping tori. Since $E$ arises from a Nielsen cycle supporting a linear stratum, the edge inclusions identify the fibre direction of the edge group with the fibre direction of exactly one of the incident vertex groups. Thus the preimages in $E$ of the centres of the two incident edge groups generate a non-cyclic, hence finite-index, subgroup. We have shown that $\cal G$ is admissible.

Having found that $G$ virtually splits as an admissible graph of groups, \cite[Thm~4]{hagenrussellsistospriano:equivariant} shows that $G$ is virtually a colourable hierarchically hyperbolic group. Though it is not explicitly stated there, the colourability can be directly seen from the construction of the hierarchy in \cite[\S5]{hagenrussellsistospriano:equivariant}, with essentially the same reasoning as in \cite[Thm~6.15]{hagenmartinsisto:extra}. 

\mk

In all three cases, $G$ is a colourable hierarchically hyperbolic group. According to \cite[Thm~B]{petyt:mapping}, every colourable hierarchically hyperbolic group is quasiisometric to a finite-dimensional CAT(0) cube complex. This concludes the proof.
\end{proof}

It is already known that all graph manifold groups are colourable hierarchically hyperbolic groups \cite{hagenrussellsistospriano:equivariant}, and also that they are quasicubical \cite{kapovichleeb:3manifold,hagenprzytycki:cocompactly}. In-keeping with this, we observe that all free-by-cyclic graph manifold groups fit into the third case of Proposition~\ref{prop:few_linear}. Indeed, by \cite[Thm~11.1]{behrstockdrutumosher:thick}, graph manifolds are \emph{thick of order 1} in the sense defined in that paper. According to \cite[Cor.~7.9]{behrstockdrutumosher:thick} and either \cite[Thm~4]{dahmanili:relative} or \cite[Thm~3.11]{ghosh:relative}, this means that free-by-cyclic graph manifold groups are polynomially growing, and hence linearly growing by \cite[Thm~1.2]{hagen:remark}. Since we are considering a manifold, no Nielsen cycle can support two linear strata.


It is natural to ask whether more of the groups considered in Proposition~\ref{prop:few_linear} are cocompactly cubulated. Results outside the hyperbolic setting seem to be fairly limited. Hagen--Przytycki characterised which graph manifolds are cocompactly cubulated \cite{hagenprzytycki:cocompactly}. Button established which \emph{tubular} groups (see Section~\ref{sec:tubular}) are free-by-cyclic \cite[Prop.~2.1]{button:tubular}, and cubulation of tubular groups is well understood \cite{wise:cubular,woodhouse:classifying:virtually}. However, it is unknown whether all toral relatively hyperbolic free-by-cyclic groups are cocompactly cubulated for example.

\mk

We now turn to the case where $\phi$ has more linear strata. Our goal will be to show that the equivalent of Proposition~\ref{prop:few_linear} fails in a strong way for many such $\phi$, by finding RBFs. We shall need the following result, which seems to be known to experts. We are grateful to Monika Kudlinska for informing us of it, and to Jean Pierre Mutanguha for sharing his preprint \cite{mutanguha:onpolynomial} with us. Recall that all free groups considered here are finitely generated.

\begin{proposition}[{\cite[Lem.~4.1,~4.2]{mutanguha:onpolynomial}}] \label{prop:undistorted}
Free-by-cyclic subgroups of free-by-cyclic groups are quasiisometrically embedded.
\end{proposition}

This includes the degenerate cases of cyclic groups and cyclic-by-cyclic groups. 

\begin{definition}[Linear $\Gamma$--path]
We refer to a path $\gamma\subset\Gamma$ as a \emph{linear $\Gamma$--path} if $\gamma$ can be decomposed as $\gamma=\gamma_1\dots\gamma_n$,  where each $\gamma_i$ is either a Nielsen path or a linear stratum.
\end{definition}

If $\gamma$ is a linear $\Gamma$--path, then $f_\#$--iterates of $\gamma$ have length that grows at most linearly.

Next we describe the mapping tori inside $G$ that arise from invariant and linear strata. For an (oriented) edge $e$ of a graph, we shall write $e^-$ for its initial vertex and $e^+$ for its terminal vertex.

\bsh{Nielsen-cycle quasiflats} \label{sh:nielsen_tori}
Let $M$ be the mapping torus of $f$, with universal cover $\tilde M$. If $p$ is a Nielsen cycle, then the mapping torus of $p$ is a $\Z^2$ subgroup of $G$, giving a quasiflat $Q\subset\tilde M$ by Proposition~\ref{prop:undistorted}. More precisely, let $\tilde p\subset\tilde M$ be a quasiline covering $p$. Pushing $\tilde p$ along fibres one step yields a path that covers $f(p)$. By tightening each lift of $f(p)$ to a lift of $f_\#(p)=p$, we obtain another quasiline covering $p$. Iterating this yields the quasiflat $Q$, which can be naturally thought of as having two axes: the ``$p$--axis'' and the ``fibre axis''. See Figure~\ref{fig:Q}.
\esh

\begin{figure}[ht]
\includegraphics[width=52mm, trim = 0 3mm 0 6mm]{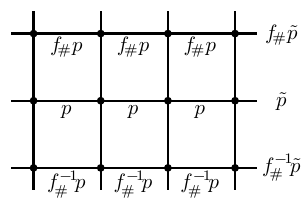}
\caption{A quasiflat $Q$ coming from the mapping torus of a Nielsen cycle $p$.} \label{fig:Q} 
\end{figure}

\bsh{Linear gluings} \label{sh:linear_gluings}
Suppose that $e$ is a linear stratum with suffix $p^n$. Let $e'$ be a subsegment of $e$ containing $e^+$ but not $e^-$. The universal cover of the mapping torus of $e'\cup p$ consists of the quasiflat $Q$ with a parallel family of strips attached. Since $f_\#(e)=ep^n$, the strips are glued to $Q$ along lines of slope $\frac1n$ with respect to the axes of Item~\ref{sh:nielsen_tori} (in the sense that moving one step in the fibre direction moves $n$ steps in the $p$ direction), as illustrated in Figure~\ref{fig:linear_gluings}. To see this when $e$ is a loop, write the fundamental group of the mapping torus of $e\cup p$ as $\sgen{p,e,t\,|\,tpt^{-1}=p,tet^{-1}=ep^n}$ to find that $e^{-1}te=p^nt$. In this case, $e$ is the stable letter of an HNN extension of $\sgen{p,t}\cong\Z^2$.

On the other hand, suppose that $e$ is a linear stratum with $e^-\in p$, a Nielsen cycle, and with suffix $p'$. In the universal cover of the mapping torus of $e\cup p\cup p'$, the quasiflat $Q$ again has a parallel family of strips attached, this time along the $f_\#$--iterates of the vertex $e^-$ of $p$. Since $e$ is a non-exponential stratum, $f(e^-)=e^-$, so these strips are glued to $Q$ along fibre lines. 
\esh

\begin{figure}[ht]
\includegraphics[width=38mm, trim = 0 3mm 0 4mm]{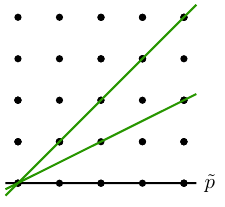}
\caption{Gluing lines in $Q$ from linear strata with suffixes $p$ and $p^2$.} \label{fig:linear_gluings} 
\end{figure}

\bsh{Branching from linear strata} \label{sh:branching}
Let $p$ be a Nielsen cycle, with corresponding quasiflat $Q\subset\tilde M$ as constructed in Item~\ref{sh:nielsen_tori}. We describe how a linear stratum $e$ with an endpoint in $p$ gives rise to half-flats branching off $Q$. Proposition~\ref{prop:undistorted} shows that these configurations are undistorted.

First suppose that $e^-\in p$, and let $p'$ be the suffix of $e$. As described in Item~\ref{sh:linear_gluings}, there are strips glued to $Q$ along the fibre lines of lifts of $e^-$. Let $E$ be such a strip. The boundary line $E^+$ not contained in $Q$ is part of a quasiflat $Q'$ corresponding to $p'$. Indeed, it is glued diagonally to $Q'$, as described in Item~\ref{sh:linear_gluings}. Taking a half-quasiflat in $Q'$ whose boundary is $E^+$ gives a half-quasiflat branching off $Q$ along the gluing line of $E$. See Figure~\ref{fig:out_flat}.

\begin{figure}[ht]
\includegraphics[width=45mm, trim = 0 5mm 0 3mm]{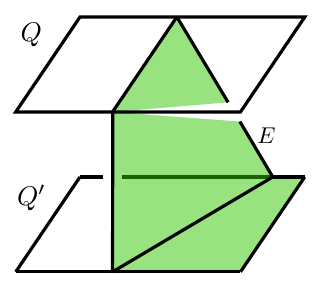}
\caption{A half-quasiflat coming from a linear stratum $e$ with $e^-\in p$.} \label{fig:out_flat} 
\end{figure}

Now suppose that $p$ supports $e$. If there is a Nielsen cycle $p'$ such that $e^-\in p'$, then the situation is exactly the reverse of the one just described, and we can take a half-quasiflat in $Q'$ whose boundary is $E^-$ to obtain a half-quasiflat branching off $Q$ along the gluing line of $E$. 

More generally, suppose that there is a linear $\Gamma$--path $\gamma$ from $e^-$ to a vertex of some Nielsen cycle $p'$. (There need not be more than one linear stratum in $\gamma$.) Since the lengths of the $f_\#$--iterates of $\gamma$ grow in length linearly, we can take a half-quasiflat $H'\subset Q'$ coming from $p'$ that is separated from $E$ by a subspace quasiisometric to the graph under the real function $x\mapsto|x|$. See Figure~\ref{fig:in_flat}. The union of these pieces is a half-quasiflat branching off $Q$ along the gluing line of $E$. Observe that, by a symmetric argument, there is also a half-quasiflat glued to $Q'$ along the terminal part of $\gamma$.

\begin{figure}[ht]
\includegraphics[width=45mm, trim = 0 4mm 0 3mm]{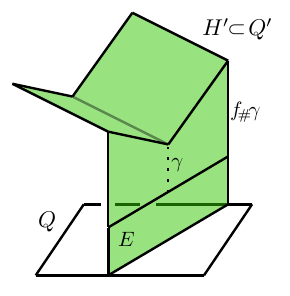}
\caption{A half-quasiflat coming from a linear stratum $e$, supported on $p$, with a linear $\Gamma$--path $\gamma$ from $e^-$ to some Nielsen cycle $p'$.} \label{fig:in_flat} 
\end{figure}

If there is no linear $\Gamma$--path from $e^-$ to a Nielsen cycle, then there is no way to extend $E$ to a half-quasiflat glued to $Q$. For example, let $F=\sgen{a,b,cdc^{-1}}$ and consider the free-by-cyclic group corresponding to the automorphism 
\[
\phi \,=\, \begin{cases} a\mapsto b, & c\mapsto cd \\ b\mapsto ba, & d\mapsto d.\end{cases}
\]
The restriction $\phi|_{\sgen{a,b}}$ gives a hyperbolic free-by-cyclic group $G_{a,b}$ by Brinkmann's theorem \cite{brinkmann:hyperbolic}. Edge-strips in $F\rtimes_\phi\Z$ corresponding to the linear stratum $c$ therefore join quasiflats coming from the Nielsen path $d$ to hyperbolic spaces corresponding to $G_{a,b}$.
\esh

In view of the above discussion, we introduce the following terminology.

\begin{definition}[Internal, source] \label{def:internal}
A linear stratum $e$ is \emph{internal} if there is a linear $\Gamma$--path from $e^-$ to a vertex of some Nielsen cycle. A vertex $v\in\Gamma$ is a \emph{source} if there is a linear stratum $e$ with $e^-=v$. We say that a Nielsen cycle \emph{has a nearby source} if there is a Nielsen path from one of its vertices to some source of $\Gamma$.
\end{definition}

By Item~\ref{sh:branching}, being internal is exactly what is needed to be able to ``extend backwards'' to a half-quasiflat, and a Nielsen cycle having a nearby source exactly means that there are half-quasiflats branching off its fibre lines. It can happen that a Nielsen cycle $p$ is the suffix of a linear stratum $e$ witnessing that a vertex of $p$ is a source, for instance if $e^-=e^+\in p$. 

\begin{definition}[Rich linearity] \label{def:rich}
We say that a free-by-cyclic group $G$ has \emph{rich linearity} if it can virtually be written as $G=F\rtimes_\phi\Z$, such that some IRTT structure of $\phi$ has a Nielsen cycle $p$ that either
\begin{itemize}
\item   supports three internal linear strata, or 
\item   supports two internal linear strata and has a nearby source.
\end{itemize}
\end{definition}

We are now in a position to prove the main result of this section. 

\begin{theorem} \label{thm:two_linear}
Free-by-cyclic groups with rich linearity do not admit coarse medians.
\end{theorem}

\begin{proof}
As $G$ has geometric dimension two, Corollary~\ref{thm:dimension} shows that if $G$ admits a coarse median, then it admits one of rank at most two. Our goal is to find a 2--RBF in $G$, for then Theorem~\ref{thm:rbf} will show that $G$ can admit no such coarse median. There is no loss in replacing $\phi$ by a positive power, because this replaces $G$ by a finite-index subgroup. We can therefore let $G=F\rtimes_\phi\sgen t$ witness rich linearity.


As described in Item~\ref{sh:nielsen_tori}, the Nielsen cycle $p$ gives rise to a quasiflat $Q$ in the universal cover of the mapping torus of $\Gamma$. Item~\ref{sh:branching} shows that each internal linear stratum supported on $p$ yields a half-quasiflat branching off $Q$ along a line not parallel to the fibre direction. Moreover, since distinct linear strata have distinct suffixes, Item~\ref{sh:linear_gluings} shows that no two of these branching lines are parallel. Item~\ref{sh:branching} also shows that if $p$ has a linear source then there is a half-quasiflat branching off $Q$ along a fibre-line. 

Thus, in either of the two cases of the statement, there are three non-parallel directions in $Q$ from which half-quasiflats branch. Since $Q$ is cocompact, this shows that $G$ contains a 2--RBF. As described above, this shows that $G$ cannot have a coarse median.
\end{proof}


\bsh{Example} \label{eg:more_than_gersten}
The following is a simple example that is not covered by Gersten's argument from \cite{gersten:automorphism} but that does have rich linearity. Let $F=\sgen{a,cbc^{-1},cd,ce}$, and consider the free-by-cyclic group corresponding to the automorphism
\[
\phi \,=\, 
\begin{cases} 
a\mapsto a,\quad c\mapsto cb, & d\mapsto da, \\
b\mapsto b, & e\mapsto ea^2.
\end{cases}
\]
\esh

We finish this section by discussing an interesting possible relation between the existence of richly branching flats and a converse to Theorem~\ref{thm:two_linear}. 

\bsh{Remark} \label{sh:converse}
Theorem~\ref{mthm:fbz_no_median} was proved by constructing RBFs, and it seems plausible that our construction is essentially the only way to do so. Indeed, if there is an RBF with base flat $B$, then the necessary quantity of branching from $B$ should only be possible in the presence of linear strata. That these linear strata branch off $B$ should then restrict $B$ to being built from Nielsen cycles, and then the richness of the branching should force the existence of a Nielsen cycle witnessing rich linearity.
\begin{itemize}
\item   Does the existence of an RBF in a free-by-cyclic group imply rich linearity?
\end{itemize}

Next, one can ask what possible obstructions there are to being able to find a coarse median on a free-by-cyclic group. To the authors' knowledge, it could be the case that the existence of RBFs is the only obstruction. Suggestively, the linear part of the geometry of a free-by-cyclic group with no quadratic strata bears similarity with that of \emph{undistorted tubular groups}, considered in Section~\ref{sec:tubular}, for which RBFs are indeed the only obstruction (Theorem~\ref{thm:tubular_special}).
\begin{itemize}
\item   If a free-by-cyclic group does not admit a coarse median, must it have an RBF?
\end{itemize}
Positive answers to the above two questions would imply the converse of Theorem~\ref{thm:two_linear}. Note that the existence of a coarse median is an invariant of the group: it does not depend on a choice of fibration.

All quasicubical metric spaces have coarse medians, but the converse is not true in general. However, in the setting of finitely generated groups, examples that distinguish the classes are sorely lacking. In many cases, a coarse median can actually be promoted to quasicubicality \cite{hagenpetyt:projection,petyt:mapping}. We expect that free-by-cyclic groups do not distinguish the classes, which explains the phrasing of Question~\ref{qn:fbz_qq}.
\begin{itemize}
\item   If a free-by-cyclic group admits a coarse median, is it necessarily quasicubical?
\end{itemize}
Assuming positive answers to all three questions, we would have a dichotomy: either $G$ has rich linearity, in which case $G$ has no coarse median; or it does not, in which case $G$ is quasicubical.
\esh

\section{Tubular groups} \label{sec:tubular}

In this section we consider \emph{tubular groups}. Our goal will be to understand which tubular groups admit coarse medians.

\begin{definition}[Tubular group]
A \emph{tubular group} is the fundamental group of a graph of groups with $\mathbf Z^2$ vertices and $\mathbf Z$ edges. 
\end{definition}

Each tubular group $G$ has an associated graph of spaces $\bar X$, with torus vertices and circle edges, such that $G=\pi_1\bar X$. We enumerate the vertex tori $\bar F_1,\dots,\bar F_n$, and the edge circles $\bar E_1,\dots,\bar E_m$. The universal cover $X$ of $\bar X$ is a tree of spaces with $2$--flat vertex spaces and line edge spaces. Given a vertex flat $F\subset X$, we shall also write $\bar F$ for the vertex torus in $\{\bar F_1,\dots,\bar F_n\}$ covered by $F$.  

\begin{definition}[Excursion decomposition]   \label{def:decomposition}
Let $F$ be a vertex flat of $X$. A path $\delta:I\to X$ is an \emph{$\bar E_j$-excursion on $F$} if it has initial and terminal segments both traversing an edge $E$ covering $\bar E_j$, and only the endpoints of $\delta$ meet $F$. A path $\gamma$ in $X$ with endpoints in some vertex flat $F$ has a unique \emph{excursion decomposition} $\gamma=\sigma_1\delta_1\sigma_2\cdots \delta_n\sigma_{n+1}$, where $\sigma_i\subset F$ and $\delta_i$ is an excursion.
\end{definition}

\begin{lemma}
\label{lem:cutPaste}
Let $\gamma$ be a geodesic in $X$ joining points in a vertex flat $F$. For each $\bar E_j$ incident to $\bar F$, there is at most one $\bar E_j$--excursion in the excursion decomposition of $\gamma$.
\end{lemma}

\begin{proof}
Let $\gamma$ be a path with two $\bar E_j$--excursions in its excursion decomposition of $\gamma$. Write $\gamma=\alpha\delta_1\beta\delta_2\eps$, where the $\delta_i$ are $\bar E_j$--excursions and $\alpha,\beta,\eps$ are subpaths of $\gamma$. There is a path $\gamma'=\alpha\delta_1\delta'_2\beta'\eps$ with the same endpoints as $\gamma$, where $\delta_2'$ and $\beta'$ are $\stab_GF$--translates of $\delta_2$ and $\beta$, respectively. We have $|\gamma'|=|\gamma|$. As $\delta_1\delta_2$ is a concatenation of $\bar E_j$--excursions, $\gamma'$ is not a geodesic, and hence nor is $\gamma$.
\end{proof}

\subsection{Distortion and isoperimetry} \label{subsec:disorted}

Recall that a subspace $Y$ of a metric space $X$ is \emph{distorted} if the inclusion map $Y\to X$ is not a quasiisometric embedding when $Y$ is given the (discrete) path metric. Let us say that a tubular group is \emph{distorted} if one of its vertex groups is distorted. We shall prove that distorted tubular groups have super-quadratic Dehn functions. 

Let $\bar E\in\{\bar E_1,\dots,\bar E_m\}$, and let $F$ cover a vertex torus to which $\bar E$ is incident. An \emph{$\bar E$--line} in $F$ is the image in $F$ of an incident edge space $E$ covering $\bar E$. An $\bar E$--line is a biinfinite $F$--geodesic, but may be distorted in $X$. Note that, for fixed $\bar E$, any two $\bar E$--lines are isometric in $X$. An \emph{$\bar E$--segment} is a segment in an $\bar E$--line.

For a path $p$ in $X$, write $\i p$ and $\t p$ for the initial and terminal vertices of $p$, respectively. 

\begin{lemma}
\label{lem:distortedLine}
If $G$ is a distorted tubular group, then there exists a distorted $\bar E$--line in $X$.
\end{lemma}

\begin{proof}
Let $F\subset X$ be a distorted flat. There is a sequence $(\gamma_i)$ of $X$--geodesics with endpoints in $F$ such that $\frac{|\gamma_i|_X}{\dist_F(\i\gamma_i,\t\gamma_i)}\to 0$. For each $i$, let $\gamma_i=\sigma^i_1\delta^i_1\dots \delta^i_{n_i}\sigma^i_{n_i+1}$ be the excursion decomposition of $\gamma_i$. By Lemma~\ref{lem:cutPaste}, we have $n_i\le m$ for all $i$. 

Since subpaths $\gamma_i$ are geodesics, we have $|\gamma_i|_X=\sum_j|\sigma^i_j|_F+\sum_j|\delta^i_j|_X$. In particular, the above convergence implies that $\sum_j|\sigma^i_j|_F<\frac12\dist_F(\i\gamma_i,\t\gamma_i)$ for all sufficiently large $i$. By the triangle inequality, $\dist_F(\i\gamma_i,\t\gamma_i)\le\sum_j|\sigma^i_j|_F+\sum_j\dist_F(\i\delta^i_j,\t\delta^i_j)$. Thus, for each sufficiently large $i$ there is some $k_i$ such that $\dist_F(\i\delta^i_{k_i},\t\delta^i_{k_i})\ge\frac1{2m}(\sum_j|\sigma^i_j|_F+\sum_j\dist_F(\i\delta^i_j,\t\delta^i_j))$. After passing to a subsequence and relabelling the $\bar E_j$, we may assume that $k_i=k$ and $\delta^i_k$ is an $\bar E_k$--excursion for all $i$. But now we compute
\[
\frac{|\delta^i_k|_X}{\dist_F(\i\delta^i_k,\t\delta^i_k)} 
    \,\le\, \frac{|\gamma_i|_X}{\dist_F(\i\delta^i_k,\t\delta^i_k)}
    \,\le\, \frac{2m|\gamma_i|_X}{\sum_j|\sigma^i_j|_F+\sum_j\dist_F(\i\delta^i_J,\t\delta^i_j)}
    \,\le\, \frac{2m|\gamma_i|_X}{\dist_F(\i\gamma_i,\t\gamma_i)} \,\to\, 0,
\]
and we conclude that $\bar E_k$--lines are distorted.
%
\end{proof}

We now define a labelled, directed graph that encodes the distortion caused by edge spaces. For each $i$, fix a basis for the vertex group $\stab_G\bar F_i$, and let $|\cdot|_i$ be the corresponding word norm.

\begin{definition}[Distortion graph]
The vertex set of $\Delta$ is the set of flats $\bar F_i$. There is a directed edge $e$ from $\bar F_{i_1}$ to $\bar F_{i_2}$ if and only if there is some distorted edge $\bar E_j$ whose endpoints are $\bar F_{i_1}$ and $\bar F_{i_2}$. This directed edge $e$ is given a label $\ell_e$ as follows. There exist $w_k\in\stab_G\bar F_{i_k}$ such that the edge $\bar E_j$ identifies $w_1$ with $w_2$. Set $\ell_e=\frac{|w_2|_{i_2}}{|w_1|_{i_1}}$.
\end{definition}

Note that all labels are rational. Also, if $e=uv$ is an edge of $\Delta$, then there is another edge $e'=vu$, with $\ell_{e'}=\frac1{\ell_e}$. We say that a directed cycle in $\Delta$ is \emph{balanced} if the product of the labels of its edges is 1.

\begin{lemma} \label{lem:balanced_undistortion}
If $G$ is distorted, then $\Delta$ has an unbalanced cycle.
\end{lemma}

\begin{proof}
Suppose that all cycles in $\Delta$ are balanced. By proceeding along a spanning tree of $\Delta$, one can label the vertices of $\Delta$ with positive integers $N_i$ in such a way that for each directed edge $e\subset \Delta$ from $\bar F_{i_1}$ to $\bar F_{i_2}$ we have $\ell_e=\frac{N_{i_2}}{N_{i_1}}$. 

We construct a metric space $X'$ with underlying set $X$. For points $x$ and $y$ in a flat covering $\bar F_i$, let $D(x,y)=N_i\dist_F(x,y)$. Set the thickness of the edge-strips to be 1. Let $X'=(X,\dist')$ be the path-metric space induced from the partially-defined function $D$. Note that by the choice of the $N_i$, every vertex-flat of $X'$ is convex, and in particular undistorted. On the other hand, given an $X$--geodesic between points $x$ and $y$ of a flat $F$, by viewing it as a union of segments in edge-strips and segments in vertex-flats we see that it is a coarsely Lipschitz path in $X'$. Hence $\dist'(x,y)$ is coarsely bounded above by $\dist(x,y)$. But this shows that $\dist(x,y)$ is coarsely lower-bounded by $\dist_F(x,y)$. We obtain a contradiction, as we have shown $X$ is quasiisometric to $X'$, whose flats are obviously undistorted.
\end{proof}

To facilitate understanding the Dehn function of $G$, we first consider the case where each flat of $X$ has at most one parallelism class of distorted $\bar E$--lines. In general this is not the same as there being only one $j$ for which there are distorted $\bar E_j$--lines, but the graph $\Delta$ can still be used to find that some $\bar E$--lines are highly distorted.

\begin{proposition} \label{prop:parallel_exponential}
Let $G$ be a distorted tubular group. If each flat of $X$ has at most one parallelism class of distorted lines, then $G$ has exponential Dehn function.
\end{proposition}

\begin{proof}
By Lemma~\ref{lem:balanced_undistortion}, there is an unbalanced cycle in $\Delta$, and hence an unbalanced embedded cycle $\gamma\subset\Delta$. Let $\alpha$ denote the product of the labels on the edges of $\gamma$. Perhaps after reversing $\gamma$, we have $\alpha<1$. Since there is only one parallelism class of distorted lines in each flat, there must be some vertex-flat $\bar F$ of $\gamma$ in which the incoming edge-gluing and the outgoing edge-gluing are (conjugates of) distinct powers of the same group element, with the outgoing one being a proper power. In $X$, this ensures that there are two conjugate edge-strips glued along the same outgoing $\bar E$--line. See Figure~\ref{fig:BScover}. After relabelling, the edges of $\gamma$ are $\bar E_1,\dots,\bar E_n$, and the initial flat is $\bar F$.

\begin{figure}[ht] 
\includegraphics[width=8cm, trim = 0 6mm 0 4mm]{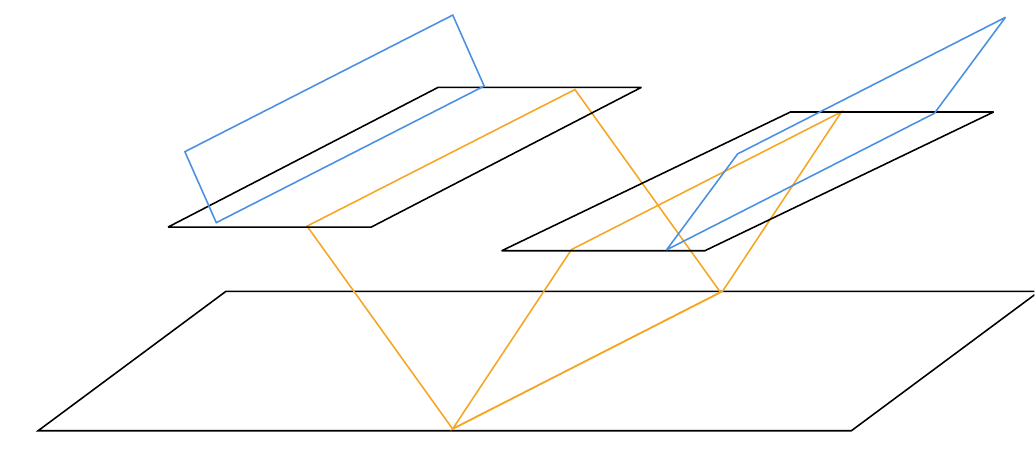} 
\caption{Two conjugate edge-strips in orange glued along an $\bar E$--line.} \label{fig:BScover} 
\end{figure}

Let $x$ and $y$ be sufficiently far-apart points lying in an $\bar E_1$--line inside a flat $F$ of $X$ covering $\bar F$. For each $n\ge0$, our choice of $\bar F$ means that $\gamma^n$ gives at least two sequences of vertex flats of $X$, starting with $F$ and ending with $F_{1,n}$ and $F_{2,n}$, respectively, which are translates of $F$. Moreover, the paths in the Bass--Serre tree of $G$ from $F$ to the $F_{i,n}$ meet only in $F$.  Among all points in $F_{i,n}$ lying in $\bar E_1$--lines, let $z_{i,n}^-$ be a closest such point to $x$. Define $z_{i,n}^+$ similarly with $y$. In particular, $z_{i,0}^-=x$ and $z_{i,0}^+=y$.

Let $p_{i,n}$ be the path from $x$ to $y$ consisting of: a geodesic $\delta_{i,n}^-$ from $x$ to $z_{i,n}^-$; the affine path in $F_{i,n}$ from $z_{i,n}^-$ to $z_{i,n}^+$; and a geodesic $\delta_{i,n}^+$ from $z_{i,n}^+$ to $y$. The length of $p_{i,n}$ is $2n+\alpha^n\dist_F(x,y)$. When $n\sim \log\dist_F(x,y)$, this is coarse-linearly equivalent to $\log\dist_F(x,y)$. 

Together, $p_{1,n}$ and $p_{2,n}$ form a loop of length coarse-linearly equivalent to $\alpha^n\dist_F(x,y)$, and they bound an embedded disc with at least $\dist_F(x,y)$ 2--cells. See Figure~\ref{fig:BSdiagram}. 

\begin{figure}[ht] 
\includegraphics[width=13cm, trim = 0cm 1cm 0cm 12mm]{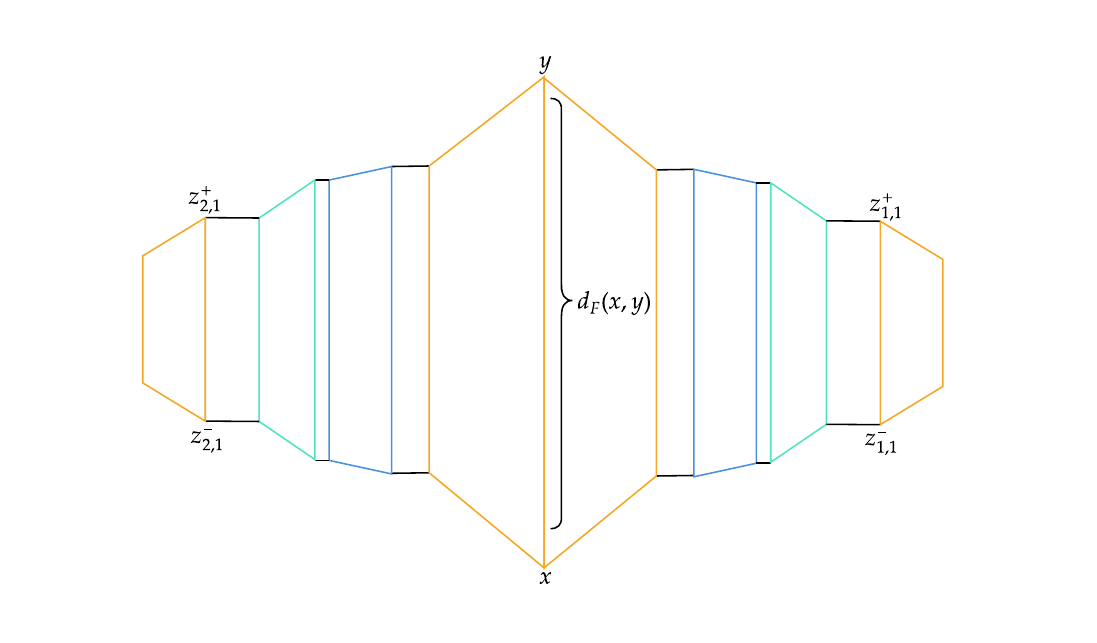} 
\caption{A example diagram where $\gamma$ is length three and $n=1$.} \label{fig:BSdiagram} 
\end{figure}
Because $X$ is a contractible 2--complex, the fact that the disc is embedded means that any disc with the same boundary must have at least as many 2--cells. In particular, by considering diagrams with $\dist_F(x,y)\to~\infty$, we find that $G$ has exponential Dehn function.
\end{proof}

\begin{theorem} \label{thm:distortion}
Distorted tubular groups have super-quadratic Dehn function.
\end{theorem}

\begin{proof}
Let $G$ be a distorted tubular group. By Lemma~\ref{lem:distortedLine}, $X$ has distorted $\bar E$-lines. If each flat of $X$ has at most one parallelism class of distorted $\bar E$--lines, then Proposition~\ref{prop:parallel_exponential} shows that $G$ has exponential Dehn function. Otherwise there is a flat $F\subset X$ with non-parallel, distorted $\bar E_1$-- and $\bar E_2$--lines. 

For $j=1,2$, let $(\delta^j_i)$ be a sequence of increasingly long $X$--geodesics that are $\bar E_j$--excursions from $F$ and witness the distortion of $\bar E_j$. Let $\sigma^j_i$ be the affine path in $F$ from $\i\delta^j_i$ to $\t\delta^j_i$. For each $i$, one can form a rhombus in $R_i\subset F$ from two translated copies of $\sigma^1_i$ and two translated copies of $\sigma^2_i$. The area of $R_i$ is quadratic in its perimeter $P_i=2|\sigma^1_i|+2|\sigma^2_i|$. By the choice of the $\delta^j_i$, this is super-quadratic in $|\delta^1_i|+|\delta^2_i|$

Taking the same translations of the $\delta^j_i$ gives a loop $L_i$ in $X$ meeting $R_i$ in exactly four points. Since $X$ is a contractible 2--complex, any disc filling $L_i$ must contain the embedded disc $R_i$, and hence have area that is super-quadratic in its perimeter. Thus $G$ has no quadratic isoperimetric function.
\end{proof}

By contrast, it is easy to see that if all vertex groups of a tubular group $G$ are undistorted, then $G$ has quadratic Dehn function. If $G$ has no distorted elements, then it has no distorted vertex groups.


\subsection{Coarse median tubular groups}

Here we use the results of Section~\ref{subsec:disorted} together with the construction of RBFs in undistorted tubular groups to prove Theorem~\ref{mthm:tubular_special}, characterising which tubular groups have coarse medians. 

\begin{lemma} \label{lem:three_edges}
Let $G$ be an undistorted tubular group. If some vertex group has three commensurability classes of incident edge groups, then $G$ has a quasiisometrically embedded 2--RBF.
\end{lemma}

\begin{proof}
Let $X$ be the tree of spaces for $G$, and let $F$ be a flat stabilised by a vertex group as in the assumption. Let $\bar E_1,\bar E_2,\bar E_3$ be the images of three pairwise non-commensurable incident edges groups in $\bar F$. For each $i$ and each line $E_i\subset F$ covering $\bar E_i$, there exists a rough half-flat attached to $F$ along $E_i$ which is the union of an edge-strip together with a half-flat in the vertex-space on the other end of the strip. As $G$ is undistorted, this yields a $2$--RBF in $X$, and hence in $G$.
\end{proof}

\begin{theorem} \label{thm:tubular_special}
Let $G$ be a tubular group. If $G$ admits a coarse median, then $G$ is cocompactly cubulated and virtually compact special.
\end{theorem}

\begin{proof}
According to \cite[Cor.~8.3]{bowditch:coarse}, if $G$ admits a coarse median then it has a quadratic isoperimetric function, and so Theorem~\ref{thm:distortion} shows that $G$ must be undistorted. Moreover, $G$ has geometric dimension two since it is a graph of two dimensional groups (see \cite[Prop.~3.6]{scottwall:topologicalmethods}), so Corollary~\ref{thm:dimension} shows that $\rk G\le2$. In particular, Theorem~\ref{thm:rbf} implies that $G$ cannot have a quasiisometrically embedded 2--RBF. By Lemma~\ref{lem:three_edges}, this means that no vertex group of $G$ can have more than two commensurability classes of incident edge groups. Since $G$ is undistorted, it contains no Baumslag--Solitar subgroups $\operatorname{BS}(m,n)$ with $m\ne\pm n$. Consequently, \cite[Cor~5.10,~5.9]{wise:cubular} implies that $G$ is cocompactly cubulated and virtually compact special.
\end{proof}

Using Wise's characterisation \cite{wise:cubular}, there are tubular groups that are freely cubulable but do not admit coarse medians.

\begin{example}[A freely cubular, tubular, non-quasicubular group]
\label{eg:c6tubular}
Let $T$ be the $1$--skeleton of a tetrahedron, and 3--edge-colour $T$ with colours $g,p,b$. Up to an isomorphism of $T$, there is one way to do this. Consider the tubular group $G$ on $T$ where $G_{e^\pm}$ is generated by: $(1,0)$ if $e$ has colour $g$; $(0,1)$ if $e$ has colour $p$; and $(1,-1)$ if $e$ has colour $b$.

It is easy to see from Wise's characterisation that $G$ acts freely on a $\mathrm{CAT}(0)$ cube complex. From work of Woodhouse \cite{woodhouse:classifying:virtually}, it follows that $G$ is virtually special. 

Moreover, the choice of edge-inclusions makes $X$ a CAT(0) space. However, since $G$ is undistorted and contains a vertex with three incident, non-commensurable edge groups, Lemma~\ref{lem:three_edges} implies $G$ has no coarse median (and in particular is not quasicubical). Groups such as $G$ can be viewed as ``opposite'' to hyperbolic groups with property (T), which are quasicubical yet do not act freely on any $\mathrm{CAT}(0)$ cube complex.

It is additionally possible to give a $C(6)$ structure to the graph of spaces associated to $G$. Each vertex space is a torus formed from a four-by-four grid of hexagons by identifying appropriate boundary edges. We then subdivide the edges, turning each hexagon to a dodecagon. Each edge annulus is formed from four hexagons, with edges subdivided. The attaching maps at each vertex are indicated in Figure~\ref{fig:c6tubular}. It can be verified by inspection that the cell structure is $C(6)$. 
\end{example}

\begin{figure}[ht] 
\includegraphics[width=9cm, trim = 0 3mm 0 3mm]{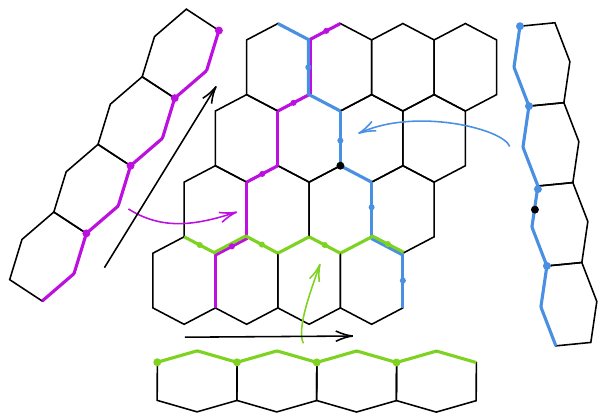} 
\caption{A $C(6)$ structure on a torus and three incident edge groups. The torus is a quotient of the hexagonal plane by translations along the black arrows, with each edge subdivided. The four-by-four hexagonal grid is a fundamental domain. The green, purple, and blue edge groups can be identified with the subgroups $\langle(1,0)\rangle$, $\langle(0,1)\rangle$, and $\langle(1,-1)\rangle$, respectively. The attaching maps send the bold vertices to bold vertices.} \label{fig:c6tubular} 
\end{figure}

There is some overlap between Theorem~\ref{thm:tubular_special} and Theorem~\ref{thm:two_linear}. By combining \cite[Cor.~4.17]{behrstockdrutu:divergence} and \cite[Thm~1.2]{macura:detour}, it can be seen that any tubular group that is free-by-cyclic is linearly growing. For such groups, one should then compare Proposition~\ref{prop:few_linear} and Theorem~\ref{thm:two_linear} with \cite[Cor.~5.10]{wise:cubular}. Wise's results and the interplay between free-by-cyclic and tubular groups have also been exploited in \cite{wuye:some} to show that certain CAT(0) free-by-cyclic groups produced by Lyman \cite{lyman:some} are virtually special but not virtually cocompactly cubulated.

\bibliographystyle{alpha}
\bibliography{bibtex}
\end{document}